\documentclass[]{interact}

\usepackage[caption=false]{subfig}

\usepackage[natbibapa,nodoi]{apacite}
\theoremstyle{plain}
\newtheorem{theorem}{Theorem}
\newtheorem{lemma}{Lemma}
\newtheorem{corollary}{Corollary}

\theoremstyle{definition}

\theoremstyle{remark}

\usepackage[utf8]{inputenc}

\usepackage[mathlines]{lineno}

\usepackage{amsmath}	

\usepackage{etoolbox}          

\newcommand*\linenomathpatch[1]{%
  \cspreto{#1}{\linenomath}%
  \cspreto{#1*}{\linenomath}%
  \csappto{end#1}{\endlinenomath}%
  \csappto{end#1*}{\endlinenomath}%
}
\newcommand*\linenomathpatchAMS[1]{%
  \cspreto{#1}{\linenomathAMS}%
  \cspreto{#1*}{\linenomathAMS}%
  \csappto{end#1}{\endlinenomath}%
  \csappto{end#1*}{\endlinenomath}%
}

\expandafter\ifx\linenomath\linenomathWithnumbers
  \let\linenomathAMS\linenomathWithnumbers
  \patchcmd\linenomathAMS{\advance\postdisplaypenalty\linenopenalty}{}{}{}
\else
  \let\linenomathAMS\linenomathNonumbers
\fi

\linenomathpatch{equation}
\linenomathpatchAMS{gather}
\linenomathpatchAMS{multline}
\linenomathpatchAMS{align}
\linenomathpatchAMS{alignat}
\linenomathpatchAMS{flalign}

\makeatletter
\patchcmd{\mmeasure@}{\measuring@true}{
  \measuring@true
  \ifnum-\linenopenaltypar>\interdisplaylinepenalty
    \advance\interdisplaylinepenalty-\linenopenalty
  \fi
  }{}{}
\makeatother

\usepackage{amssymb}	
\usepackage{amsfonts}	
\usepackage{amsthm}	
\usepackage{sansmath}

\usepackage{graphicx,xcolor}

\usepackage{enumitem}

\newcommand{\bp}{\boldsymbol{p}}
\renewcommand{\bm}{\boldsymbol{m}}
\newcommand{\ba}{\boldsymbol{a}}

\newcommand{\bC}{\boldsymbol{C}}

\newcommand{\bu}{\boldsymbol{u}}

\newcommand{\bv}{\boldsymbol{v}}
\newcommand{\bV}{\boldsymbol{V}}
\newcommand{\bk}{\boldsymbol{k}}
\renewcommand{\bi}{\boldsymbol{i}}
\newcommand{\bx}{\boldsymbol{x}}
\newcommand{\bw}{\boldsymbol{w}}

\newcommand{\sT}{\mathsf{T}}
\newcommand{\sP}{\mathsf{P}}
\newcommand{\sQ}{\mathsf{Q}}
\newcommand{\sI}{\mathsf{I}}
\newcommand{\sJ}{\mathsf{J}}
\newcommand{\sA}{\mathsf{A}}

\newcommand{\sM}{\mathsf{M}}
\newcommand{\sS}{\mathsf{S}}
\newcommand{\bR}{\boldsymbol{R}}

\newcommand{\bOmega}{\boldsymbol{\Omega}}
\newcommand{\bgamma}{\boldsymbol{\gamma}}
\newcommand{\Eta}{\boldsymbol{\eta}}
\newcommand{\bxi}{\boldsymbol{\xi}}

\newcommand{\grad}{\boldsymbol{\nabla}}
\newcommand{\eps}{\varepsilon}
\newcommand{\Chi}{\boldsymbol{\chi}}
\newcommand{\bbR}{\mathbb{R}}
\newcommand{\bbT}{\mathbb{T}}
\newcommand{\cD}{\mathcal{D}}

\newcommand*{\Fr}{{\operatorname{Fr}}}
\newcommand*{\Bu}{{\operatorname{Bu}}}

\newcommand{\Range}{\operatorname{Range}}
\newcommand{\Ker}{\operatorname{Ker}}
\newcommand{\curl}{\operatorname{curl}}
\newcommand{\Vdiv}{V_{\operatorname{div}}}
\newcommand{\Diff}{\operatorname{Diff}}

\allowdisplaybreaks

\title{Variational balance models for the three-dimensional
Euler-Boussinesq equations with full Coriolis force}

\author{
\name{G\"ozde \"Ozden\textsuperscript{a,b,c} and Marcel Oliver\textsuperscript{a}}
\affil{\textsuperscript{a}School of Engineering and Science, Jacobs
University, 28759 Bremen, Germany; \\ \textsuperscript{b}MARUM -- Center
for Marine Environmental Sciences, Universität Bremen, 28334 Bremen,
Germany;
\textsuperscript{c}Meteorological Institute, Universität Hamburg,
20146 Hamburg, Germany}
\footnote{Corresponding author: Gözde Özden, goezde.oezden@uni-hamburg.de}
}

\begin{document}

\maketitle
\begin{abstract}
We derive a semi-geostrophic variational balance model for the
three-dimensional Euler--Boussinesq equations on the non-traditional
$f$-plane under the rigid lid approximation.  The model is obtained by
a small Rossby number expansion in the Hamilton principle, with no
other approximations made.  We allow for a fully non-hydrostatic flow
and do not neglect the horizontal components of the Coriolis
parameter, i.e., we do not make the so-called ``traditional
approximation''.  The resulting balance models have the same structure
as the ``$L_1$ balance model'' for the primitive equations: a
kinematic balance relation, the prognostic equation for the
three-dimensional tracer field, and an additional prognostic equation
for a scalar field over the two-dimensional horizontal domain which is
linked to the undetermined constant of integration in the thermal wind
relation.  The balance relation is elliptic under the assumption of
stable stratification and sufficiently small fluctuations in all
prognostic fields.
\end{abstract}


\section{Introduction}
\label{s.intro}

We describe the derivation of a semi-geostrophic variational balance
model for the three-dimensional Euler--Boussinesq equations on the
non-traditional $f$-plane under the rigid lid approximation.  In
non-dimensional variables, the Euler--Boussinesq system takes the form
\begin{subequations}
\label{e.non-boussinesq}
\begin{gather}
  \eps \, (\partial_t \bu + \bu \cdot \grad \bu)
  + \bOmega \times \bu = - \grad p - \rho \, \bk \,,
  \label{e.non-boussinesq.a} \\
  \grad \cdot \bu = 0 \,, \\
  \partial_t \rho + \bu \cdot \grad \rho = 0 \,. \label{e.non-boussinesq.c}
\end{gather}
\end{subequations}
where $\eps=U/(fL)$ is the Rossby number (here, $U$ denotes a typical
horizontal velocity scale, $L$ a typical horizontal length scale, and
$f$ the Coriolis frequency in physical units), assumed small, $\bu$ is
the three-dimensional velocity field,
\begin{gather}
  \bOmega
  = \begin{pmatrix}
      0 \\ \cos{\phi} \\ \sin{\phi}
    \end{pmatrix}
  \equiv
  \begin{pmatrix}
    0 \\ c \\ s
  \end{pmatrix}
\end{gather}
the full Coriolis vector at constant latitude $\phi$ which, without
loss of generality, is assumed to lie in the $y$-$z$ plane as shown in
Figure~\ref{f.fig}, $p$ is the pressure, $\rho$ the density, and $\bk$
the unit vector in the vertical.  See, e.g., \cite{Maj03} or
\cite{vallis2017atmospheric} for details on the Euler--Boussinesq
equations and \cite{FranzkeOR:2019:MultiSM} for an explicit exposition
of the semi-geostrophic scaling limit.

\begin{figure}[tb]
\centering
\includegraphics[width=0.8\textwidth]{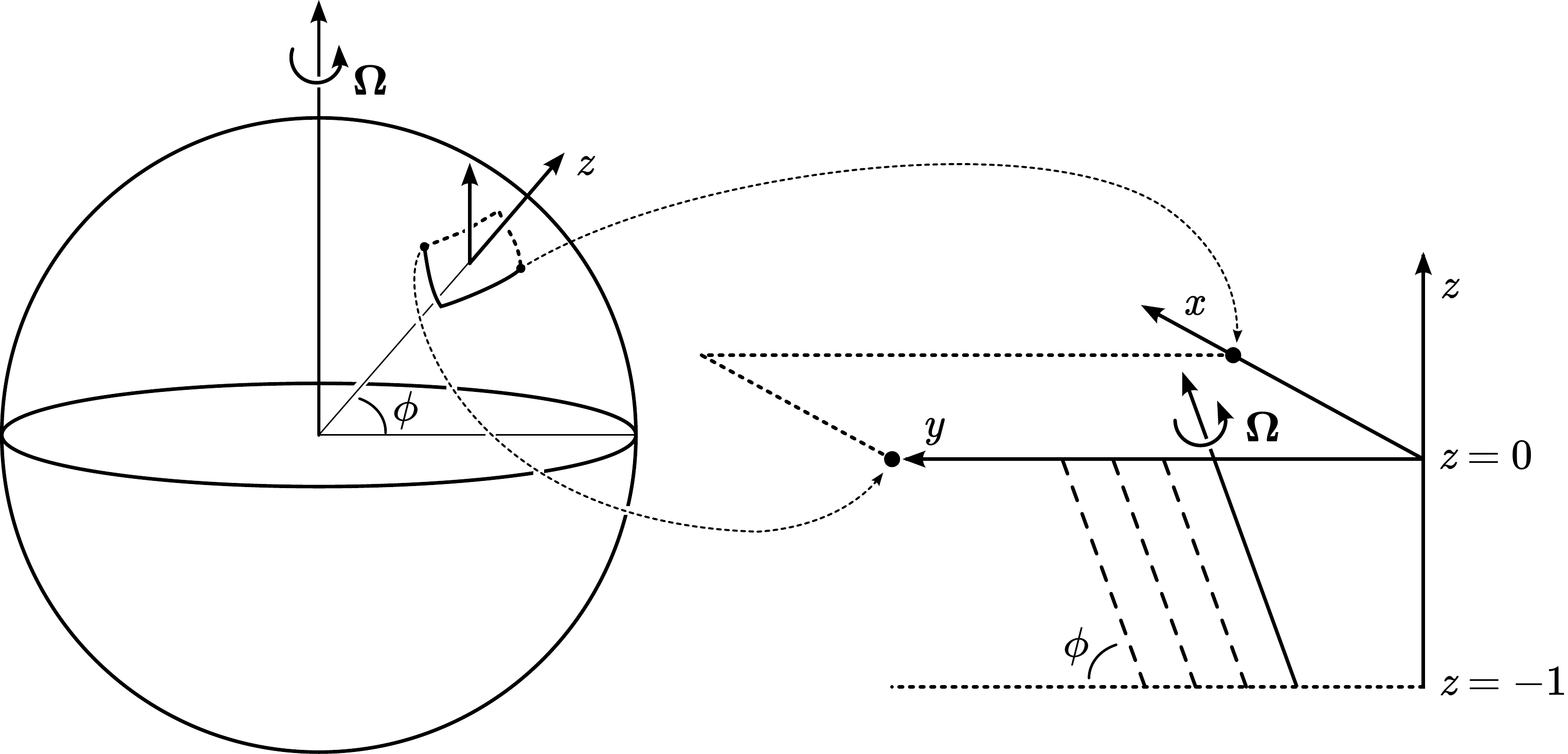}
\caption{Geometry of the $f$-plane approximation at latitude $\phi$
with the full Coriolis vector. Here, $x$, $y$, and $z$ are directed
towards the east, the north, and upward, respectively.  The dashed
lines parallel to the axis of rotation in the $y$-$z$ plane represent
the characteristics of the thermal wind relation.}
\label{f.fig}
\end{figure}
For simplicity, we assume periodic boundary conditions in
the horizontal and rigid lid boundary conditions in the vertical on a
layer of fluid with constant depth,
\begin{equation}
  \cD = \bbT^2 \times [-1,0] \,,
\end{equation}
so that
\begin{equation}
  \bk \cdot \bu = 0 \quad \text{at} \quad z=-1, 0 \,.
\end{equation}

The semi-geostrophic scaling used in \eqref{e.non-boussinesq}
corresponds to a regime of strong rotation and weaker stratification,
expressed by a Burger number $\Bu=\eps$ (or, equivalently,
$\eps = \Fr^2$, where $\Fr$ denotes the Froude number).  In the more
widely studied quasi-geostrophic regime, rotation and stratification
are equally important so that the Burger number is $O(1)$.  The
quasi-geostrophic approximation must be made in conjunction with the
assumption that the density is a small perturbation of a constant,
stably stratified background density.  The quasi-geostrophic limit as
it is found in textbooks
\cite[e.g.][]{Pedlosky:1987:GeophysicalFD,vallis2017atmospheric} is commonly studied
under the additional assumption of the hydrostatic and traditional
approximations and is widely used for proof-of-concept studies.
However, \cite{Embid-1998} have shown that it is possible to derive
quasi-geostrophic limit equations without hydrostaticity and with the
full Coriolis vector; this limit has a rigorous justification within
the framework of averaging over fast scales developed in
\cite{EmbidM:1996:AveragingFG}.  \cite{JulKnoMil06} systematically
explore generalized quasi-geostrophic models by allowing, in contrast
to the earlier work of \cite{Embid-1998}, different domain aspect
ratios and different tilting regimes of the axis of rotation; see
\cite{LucasMR:2017:NontraditionalQG} for recent rigorous analysis and
\cite{NievesGJ:2016:InvestigationsNH} for a numerical study.  We also
refer the reader to \cite{BabinMN:2002:FastSL} for a detailed
discussion of the different scaling regimes for rotating stratified
flow and for a discussion of resonances when averaging over fast
waves.  A more expository survey can be found in
\cite{FranzkeOR:2019:MultiSM}.

In our setting, there is no need to split off a small perturbation
density from a stably stratified background profile \emph{a priori}.
However, we shall see that this condition does not disappear
altogether but comes back in slightly weaker form as a solvability
condition for the balance relation.  By keeping the density as a
Lagrangian tracer, it is possible to retain the variational
formulation of the equation of motion throughout the limit.  

As in \cite{Embid-1998}, \cite{JulKnoMil06} and
\cite{LucasMR:2017:NontraditionalQG}, we use the full Coriolis vector,
in contrast to the more common ``traditional approximation'' where
only the vertical component of the Coriolis vector is retained---a
consistency requirement if the hydrostatic approximation is made
\cite[cf.][]{White:2002:ViewEM}.  Vice versa, when the aspect ratio,
the ratio of typical horizontal and typical vertical scales is of
order one so that the model is fully non-hydrostatic, as is assumed
here, both horizontal and vertical components of the Coriolis vector
contribute to the leading order on the mid-latitude $f$-plane.
\cite{WhiWin14} perform a more idealized numerical study for the
non-hydrostatic Boussinesq equations with only traditional Coriolis
forces (a ``polar $f$-plane'') in three different limits: the
quasi-geostrophic regime, the strongly stratified regime where the
Rossby number remains $O(1)$ while the Froude number goes to zero
\citep[both limits as considered by][]{Embid-1998} as well as the case
of strong rotation and weak stratification where the Froude number
remains $O(1)$ while the Rossby number goes to zero.
\cite{WhiteheadHW:2018:EffectTD} also consider intermediate regimes
between quasi-geostrophy and the weak stratification limit, such as
the semi-geostrophic limit considered here, within the theoretical
framework of \cite{EmbidM:1996:AveragingFG}.  More recent related
theoretical results are due to \cite{JuM:2019:LowFR}.
\cite{WetzelSS:2019:BalancedUC} perform a numerical study of balance
for a moist atmosphere with phase changes, and
\cite{KafiabadB:2018:SpontaneousIN,KafiabadB:2017:RotatingST,KafiabadB:2016:BalanceDR}
and \cite{KafiabadVY:2021:WaveAB} study the exchange of energy between
balanced and unbalanced motion.

Our derivation is based on variational asymptotics.  Approximations to
Hamilton's variational principle for the equations of rotating fluid
flow were pioneered by \cite{Salmon83,Salmon1985} in the context of
the shallow water equations and later, in \cite{Salmon1996}, for the
primitive equations of the ocean.  The basic idea is to consider the
``extended Hamilton principle'' or ``phase space variational
principle'' and use the leading order balance relation (geostrophic
balance for the shallow water equations or the thermal wind relation
for the primitive equations) to constrain the momentum variables.  The
stationary points of the constrained action with respect to variations
of the position variables give a balance relation that includes the
next-order correction to the leading order constraint.  In principle,
this method can be iterated to obtain higher-order balance relations.
For example, \citet{AllenHN2002TowardEG} derive second order models
and show that the so-called $L_1$ and $L_2$ models are numerically
well behaved.

A second idea, also proposed in \cite{Salmon1985}, is the use of a
near-identity coordinate transformation to bring the resulting
variational principle or the equations of motion into a more
convenient form.  This transformation may be applied perturbatively,
so that the resulting models coincide only up to terms of a certain
order.  From an analytic perspective, higher-order terms may matter
and it turns out that a transformation to a coordinate system in which
the Hamilton equations of motion take canonical form, the original
motivation behind this approach, leads to ill-posedness of the full
system of prognostic equations.

An even more general framework is obtained by reversing the steps
``constrain'' and ``transform''.  \citet{Oliver:2006:VariationalAR}
noted that is possible to assume an entirely general near-identity
change of coordinates and expand the transformed Lagrangian in powers
of the small parameter.  Whenever the transformation is chosen such
that the Lagrangian is degenerate to the desired order, the
variational principle \emph{implies} a phase space constraint.  This
approach gives rise to a greater variety of candidate models that are
not accessible via Salmon's approach.  In particular, it allows us to
retain some important mathematical features such as regularity of
potential vorticity inversion.  In the context of the shallow water
equations on the $f$-plane, it turns out that the $L_1$ model first
proposed by Salmon is already optimal among a larger family of models
\citep{DritschelGO:2016:ComparisonVB}.  In other configurations, such as
when the Coriolis parameter is spatially varying \citep{Oliver:2013},
or in the fully three dimensional situation considered here, the
additional flexibility that comes from putting the change of
coordinates first is necessary.

While our scaling is consistent with the geostrophic momentum
approximation \citep{eliassen1948quasi}, the derivation and the
resulting equations are different.  \cite{hoskins1975geostrophic} has
shown that the geostrophic momentum approximation can be formulated as
potential vorticity advection in transformed coordinates, known as the
semigeostrophic equations.  His transformation has subsequently been
interpreted as a Legendre transformation, providing a sense of
generalized solutions via the theory of optimal transport
\citep{cullen1984extended,benamou1998weak,Cullen:2006:MathematicalTL,
Colombo:2017:FlowsNS,RoulstoneS:1997:MathematicalST}.  On the other
hand, $L_1$-type models, at least in the shallow water context, have
strong solutions in a classical sense \citep{CalikOV:2013:GlobalWG}.
Semi-geostrophic theory can be seen as a particular choice of
truncation and transformation in the framework of variational
asymptotics \citep {Oliver:2014:VariationalDG}; for a different view
on generalized semigeostrophic theory, see
\cite{McIntyreR:2002:HigherAA}.

In this work, we show that variational asymptotics in the
semi-geostrophic regime can be done directly for the three-dimensional
Euler--Boussinesq system with full Coriolis force without any
preparatory approximations.  Our motivation derives, first, from the
role of the Boussinesq equation as the common parent model of nearly
all of geophysical fluid dynamics.  Second, we would like to see how
``non-traditional'' Coriolis effects, associated with the vertical
component of the Coriolis force and horizontal Coriolis forces coming
from the vertical velocity, enter the balance dynamics.
Non-traditional effects are significant in a variety of circumstances
\citep{GerkemaS:2005:Near-inertialWN,GerkemaS:2005:Near-inertialWO,GerkemaZM:2008:GeophysicalAF,brummel96,juarez02,Hayashi-2012}
and are also studied in the context of other reduced models such as
layers of shallow water
\citep{stewart2010multilayer,stewart2012multilayer} and models on the
$\beta$-plane \citep{Dellar:2011:VariationsBP}.

Our approach is similar, in principle, to the primitive equation case
\citep{OliverV:2016:GeneralizedLS}, and we find that the resulting
balance models have the same structure: a kinematic balance relation
which is elliptic under suitable assumptions, the prognostic equation
for the three-dimensional tracer field, and an additional prognostic
equation for a scalar field over the two-dimensional horizontal domain
which is linked to the undetermined constant of integration in the
thermal wind relation.  This field, in the present setting, takes the
form of a skewed relative vorticity averaged along the axis of
rotation.

On the technical level, the computations are considerably more
difficult than for the primitive equation.  Some expressions, such as
the thermal wind relation, take a natural form in an oblique
coordinate system where the axis of rotation takes the role of the
``vertical''.  Other expressions, most notably the top and bottom
boundary conditions and the gravitational force, are most easily
described in the usual Cartesian coordinates.  This incompatibility
requires a detailed study of averaging and decomposition of vector
fields in oblique coordinates, and of the translations between oblique
and Cartesian coordinates.

The remainder of the paper is organized as
follows. Section~\ref{s.var-principle} recalls the variational
derivation for the Euler--Boussinesq equations in the language of the
Euler--Poincar\'e variational principle.
Section~\ref{s.v-asymptotics} explains the general setting for
variational asymptotics in this framework.  In
Section~\ref{s.mapping-decomp}, we set up an oblique coordinate system
whose vertical direction is aligned with the axis of rotation and
introduce the notion of averaging along this axis.
Section~\ref{s.twr} discusses the leading order balance, the thermal
wind relation, on the mid-latitude $f$-plane with full Coriolis force.
In Section~\ref{s.l1}, we derive the first order balance model
Lagrangian.  To ensure that the Lagrangian is maximally degenerate, we
carefully decompose the expression for kinetic energy into the parts
that can be removed by a suitable choice of the transformation vector
field and a residual component which cannot be removed.  In
Section~\ref{s.l1-variation}, we derive the balance model
Euler--Poincar\'e equations from the balance model Lagrangian.  In
Section~\ref{s.bal-comp}, we find that it can be decomposed into a
single evolution equation for a scalar field in the two horizontal
variables, and a kinematic relationship for all other components.  We
show that this kinematic relationship is elliptic under the assumption
of stable stratification and sufficiently small fluctuations in all
prognostic fields.  Section~\ref{s.solution} discusses the
reconstruction of the full velocity field from the balance relation
and the prognostic variables. We finally give a brief discussion and
conclusions.  Finally, four technical appendices contain results on
averaging along the axis of rotation, the decomposition of vector
field in oblique coordinates, associated inner product identities, and
some details of the computation of the balance model Lagrangian.

\section{Variational principle for the Boussinesq equations}
\label{s.var-principle}

In this section, we recall the derivation of the equations of motion
via the Hamilton principle, here in the abstract setting of
Euler--Poincar\'{e} theory.  We also set up the derivation of reduced
equations of motion via degenerate variational asymptotics.

We write $\Diff_\mu(\cD)$ to denote the group of $H^s$-class
volume-preserving diffeomorphisms on $\cD$ that leave its boundary
invariant.  The associated ``Lie algebra'' of vector fields is the
space
\begin{equation}
  \Vdiv = \{ \bu \in H^s(\cD, \bbR^3) \colon \grad \cdot \bu = 0, \, \bk
  \cdot \bu = 0 \text{ at } z=-1, 0 \} \,.
\end{equation}
Here, $H^s$ is the usual Sobolev space of order $s$ consisting of
functions with square integrable weak derivatives up to order $s$.
For $s>5/2$, $\Diff_\mu(\cD)$ is a smooth infinite-dimensional
manifold and $\Vdiv$ is its tangent space at the identity
\citep{Palais:1968:FoundationsGN, EbinM:1970:GroupsDM}.

Let $\Eta=\Eta( \, \cdot \,, t) \in \Diff_\mu(\cD)$ denote the
time-dependent flow generated by a time-dependent vector field
$\bu = \bu( \, \cdot \,, t) \in \Vdiv$, i.e.,
\begin{equation}
\label{e.flow}
  \partial_t \Eta (\ba,t) = \bu (\Eta(\ba,t),t) 
\end{equation}
or $\dot \Eta = \bu \circ \Eta$ for short; here and in the following
we use the symbol ``$\circ$'' to denote the composition of maps and
the dot-symbol to denote the \emph{partial} time derivative.  In the
Lagrangian description of fluid flow, $\bx(t) = \Eta(\ba,t)$ is the
trajectory of a fluid parcel initially located at $\ba \in \cD$.  We
retain the letter $\bx$ for Eulerian positions and $\ba$ for
Lagrangian labels throughout.

To formulate the variational principle, we note that the continuity
equation \eqref{e.non-boussinesq.c} is equivalent to
\begin{equation}
   \rho \circ \Eta = \rho_0 \,,
\end{equation}
where $\rho_0$ denotes the given initial density field.  The
Lagrangian for the non-dimensional Euler--Boussinesq system is given
by
\begin{align}
\label{eqn:lagrangian}
  L(\Eta, \dot{\Eta}; \rho_0)
  & = \int_\cD \bR \circ \Eta \cdot \dot{\Eta}
      + \frac{\eps}{2} \, \lvert \dot{\Eta} \lvert^2
      - \rho \, \bk \cdot \Eta \, d \ba
    \notag \\
  & = \int_\cD \bR \cdot \bu + \frac{\eps}{2} \, \lvert \bu \lvert^2
      - \rho \, z \, d\bx
      \equiv \ell (\bu,\rho) \,.
\end{align}
The vector $\bR$ is a vector potential for the Coriolis vector, i.e.,
$\grad \times \bR = \bOmega$; it arises from the transformation into a
rotating frame of reference \cite[e.g.][]{LandauL:1976:Mechanics}.
Here, on the $f$-plane, we can choose
\begin{equation}
    \bR = \frac{1}{2} \sJ \bx
    = \frac{1}{2} \begin{pmatrix}
        c z  - s y  \\
        s x \\
        - c x 
      \end{pmatrix} \,,
\end{equation}
where $\sJ$ is the skew-symmetric matrix
\begin{equation}
  \label{e.J}
    \sJ
    = \begin{pmatrix}
    0 & -s  & c \\
    s & 0 & 0 \\
    -c & 0 & 0
    \end{pmatrix} \,.
\end{equation}
As we see in the second line of \eqref{eqn:lagrangian}, the Lagrangian
can be written as a functional of $\bu$ and $\rho$ alone.  In this
form, it is referred to as the reduced Lagrangian $\ell$.  Computing
variations of the full Lagrangian $L$ with respect to the flow map
$\Eta$ is equivalent to computing variations of the reduced Lagrangian
$\ell$ with respect to $\bu$ and $\rho$ subject to so-called Lin
constraints. This is summarized in the following theorem.

\begin{theorem}[\citealp{HolmMR:1998:EulerPE}] \label{t.ep}
Consider a curve $\Eta$ in $\Diff_\mu(\cD)$ with Lagrangian velocity
$\dot \Eta$ and Eulerian velocity $\bu \in \Vdiv$.  Then the
following are equivalent.
\begin{enumerate}[label={\upshape(\roman*)}]
\item $\Eta$ satisfies the Hamilton variational principle
\begin{equation}
  \delta \int_{t_1}^{t_2} L(\Eta, \dot \Eta; \rho_0) \, dt = 0
\end{equation}
with respect to variations of the flow map
$\delta \Eta = \bv \circ \Eta$, where $\bv$ is a curve in $\Vdiv$
vanishing at the temporal end points.

\item $\bu$ and $\rho$ satisfy the reduced Hamilton principle 
\begin{equation}
  \delta \int_{t_1}^{t_2} \ell(\bu, \rho) \, dt = 0 
\end{equation}
with respect to variations $\delta \bu$ and $\delta \rho$ that are
subject to the Lin constraints
$\delta \bu = \dot \bv + \bu \cdot \grad \bv - \bv \cdot \grad \bu$
and $\delta \rho + \bv \cdot \grad \rho = 0$, where $\bv$ is a curve
in $\Vdiv$ vanishing at the temporal end points.

\item $\bm$ and $\rho$ satisfy the Euler--Poincar\'{e} equation
\begin{equation}
\label{e.EP-abstract}
  \int_{\cD} (\partial_t + \mathcal{L}_{\bu}) \bm \cdot \bv
    + \frac{\delta \ell}{\delta \rho} \,
      \mathcal{L}_{\bv} \rho \, d\bx = 0 
\end{equation}
for every $\bv \in \Vdiv$, where  $\mathcal{L}_{\bu}$ is the Lie
derivative in the direction of $\bu$ and
\begin{equation}
  \bm = \frac{\delta \ell}{\delta \bu}
\end{equation}
is the momentum one-form.
\end{enumerate}
\end{theorem}

In the language of vector fields on a region of $\bbR^3$,
\eqref{e.EP-abstract} reads
\begin{equation}
\label{eqn:variational3}
  \int_{\cD} \Bigl(
         \partial_t \bm + (\grad \times \bm)\times \bu
         + \grad (\bm \cdot \bu)
         + \frac{\delta \ell}{\delta \rho} \, \grad \rho
       \Bigr) \cdot \bv \, d\bx = 0 
\end{equation}
for every $\bv \in \Vdiv$.  This implies that the term in parentheses
must be zero up to a gradient of a scalar potential $\phi$.  

For the Euler--Boussinesq Lagrangian \eqref{eqn:lagrangian},
\begin{equation}
  \bm = \frac{\delta \ell}{\delta \bu} = \bR + \eps \, \bu
  \qquad \text{and} \qquad
  \frac{\delta \ell}{\delta \rho} = - z \,.
\end{equation}
Setting $\phi = - p - z \, \rho - \frac{1}{2} \, \eps \, \lvert \bu \rvert^2$
and using the vector identity
$(\grad \times \bu) \times \bu= \bu \cdot \grad \bu - \tfrac{1}{2}
\grad |\bu|^2$, we see that \eqref{eqn:variational3} implies the
Euler--Boussinesq momentum equation \eqref{e.non-boussinesq}.

As the Lagrangian $L$ is invariant under time translation, the
corresponding Hamiltonian
\begin{equation}
\label{e.hamiltonian}
  H = \int_\cD \frac{\delta \ell}{\delta \bu} \cdot \bu \, d\bx
        - \ell(\bu, \rho)
    = \int_\cD \frac{\eps}2 \, \lvert \bu \rvert^2 + z \, \rho \, d\bx
\end{equation}
is a constant of the motion. A second conservation law arises via the
invariance of the Lagrangian with respect to ``particle relabeling'',
i.e., composition of the flow map with an arbitrary time-independent
map in $\Diff_\mu(\cD)$: \cite{Ripa:1981:SymmetriesCL} and
\cite{Salmon:1982:HamiltonsPE} have shown that this symmetry implies
material conservation of the Ertel \emph{potential vorticity}
\begin{equation}
  q = (\grad \times \bm) \cdot \grad \rho
    = (\bOmega + \eps \, \grad \times \bu) \cdot \grad \rho \,,
\end{equation}
i.e.\ $q$ satisfies the advection equation
\begin{equation}
    \partial_t q + \bu \cdot \grad q = 0 \,.
\end{equation}

\section{Variational asymptotics}
\label{s.v-asymptotics}

Our variational balance model is based on the following construction.
Suppose that the flow of the balance model is related to the flow of
the full model via a near-identity change of variables.  To be
explicit, let $\dot{\Eta}_\eps$ denote the Lagrangian velocity and
$\bu_\eps$ the corresponding Eulerian velocity field of the full
Euler--Boussinesq flow, i.e.
\begin{equation}
  \dot{\Eta}_\eps = \bu_\eps \circ \Eta_\eps \,.
  \label{e.etaepsdot}
\end{equation}
The corresponding Lagrangian and Eulerian balance model velocities are
denoted $\dot{\Eta}$ and $\bu$, so that
\begin{equation}
  \dot{\Eta} = \bu \circ \Eta \,.
\end{equation}
We suppose that the flow map of the full model is related to the flow
map of the balance model by a change of coordinates that is the flow
of a vector field $\bv_\eps$ with $\eps$ as the flow parameter (thus,
formally, $\eps$ plays the same role here that $t$ plays above),
namely
\begin{equation}
  \Eta_\eps' = \bv_\eps \circ \Eta_\eps \,,
  \qquad \Eta_\eps \big\vert_{\eps = 0} = \Eta \,.
  \label{e.etaepsprime}
\end{equation}
Here and in the following, we use the prime symbol to denote a
derivative with respect to $\eps$.  Cross-differentiation of
\eqref{e.etaepsdot} and \eqref{e.etaepsprime} gives
\begin{equation}
  \bu'_\eps
  = \dot{\bv}_\eps + \bu_\eps \cdot\grad \bv_\eps
    - \bv_\eps \cdot \grad \bu_\eps \,.
    \label{e.lin-constraint}
\end{equation}
Likewise, differentiation in $\eps$ of the Lagrangian density gives
\begin{equation}
  \rho_\eps' + \bv_\eps \cdot \grad \rho_\eps = 0 \,.
  \label{e.lin-constraint-2}
\end{equation}
These two relations can be seen as the Lin constraints for the
$\eps$-flow, cf.\ the Lin constraints for the $t$-flow stated in
Theorem~\ref{t.ep}.

At this point, the choice of $\bv_\eps$ is completely arbitrary and no
assumptions have been made.  However, we will always restrict
ourselves to incompressible transformations which leave the domain
invariant, so that $\bv_\eps \in \Vdiv$.

We now treat $\eps$ as a perturbation parameter, expanding
\eqref{e.lin-constraint} and \eqref{e.lin-constraint-2} in $\eps$.
Likewise, we expand the Lagrangian as a formal power series in $\eps$
and use the Lin constraints to eliminate all $\eps$-derivatives of
$\bu$ and $\rho$.  Model reduction is achieved via the following
steps.

\begin{enumerate}[label={\upshape(\roman*)}]
\item Truncate the expansion of the Lagrangian at some fixed order.
Here, we will only look at the truncation to $O(\eps)$, the first
nontrivial case.

\item Choose the expansion coefficients of the transformation vector
fields such that the resulting Lagrangian is maximally degenerate.
When computing the model reduction to first order, the zero-order
transformation vector field $\bv = \bv_\eps\vert_{\eps=0}$ is the only
choice to be made.

\item Apply the Euler--Poincar\'e variational principle to derive the
equations of motion.  Since the Lagrangian is degenerate, some degrees
of freedom will be kinematic, thus imply a phase-space constraint that
can be understood as a so-called Dirac constraint
\citep[e.g.][]{Salmon:1988:SemigeostrophicDB}. 
\end{enumerate}

The first two steps are done in reverse order relative to the original
method of \cite{Salmon1985} who constrained the system first, using
the readily available leading order balance relation, and transformed
to different coordinates second.  In fact, in his approach, the second
step is optional and it turns out that, for $f$-plane shallow water,
the so-called $L_1$ model, which skips the transformation, is superior
to all other models in a more general family
\citep{DritschelGO:2016:ComparisonVB}.  In our viewpoint, the
transformation is always necessary.  The advantage is that the
constraint arises as a consequence of the formalism and does not need
to be guessed or derived \emph{a priori}.  In simple cases, it is
possible to choose the transformation such that it vanishes to the order
considered; in this case, we reproduce Salmon's $L_1$ model.  In more
complicated cases, in particular in our setting here, it is not
possible to cancel all terms in the transformation vector field at the
order considered.  This means that a direct application of Salmon's
method would result in a model with additional spurious prognostic
variables.  We finally remark that at least in finite dimensions, the
procedure is rigorous and the resulting model is correct to the
expected order of approximation \citep{GottwaldO:2014:SlowDD}.  For
partial differential equations, this question is open; it is clear,
though, that additional conditions are necessary.

\section{Oblique vertical coordinate and oblique averages}
\label{s.mapping-decomp}

The axis of rotation defines the direction of the characteristics of
the thermal wind relation.  Here, we introduce notation used
throughout the paper to describe the decomposition of vector fields
along and perpendicular to this axis, as well as a basic averaging
operation along the axis of notation.  

Let
\begin{gather}
  \sQ = \bOmega \bOmega^\sT
  = \begin{pmatrix}
     0 & 0 & 0 \\ 0 & c^2 & cs \\
     0 & cs & s^2
    \end{pmatrix}
  \label{e.Qdef}
\end{gather}
denote the orthogonal projector onto the direction of the axis of
rotation; the projector onto the orthogonal complement is then given
by
\begin{gather}
  \label{e.Pdef}
  \sP = \sI - \sQ
  = \begin{pmatrix}
    1 & 0 & 0 \\ 0 & s^2 & -cs \\
    0 & -cs & c^2
    \end{pmatrix} \,.
\end{gather}
Note that $\bOmega$ spans the kernel of $\sJ$ and is perpendicular to
the range of $\sJ$.  Thus, $\sQ$ is the orthogonal projector onto
$\Ker \sJ$ and $\sP$ is the orthogonal projector onto $\Range \sJ$.

We introduce the oblique coordinate system
\begin{equation}
  \bx = \begin{pmatrix}
          x \\ y \\ 0
        \end{pmatrix}
        + \zeta \, \bOmega
  \equiv \sA \bxi \equiv \Chi(\bxi)
  \label{e.mapping}
\end{equation}
where $\bxi \equiv (x,\zeta)^\sT$, where $x$ and $y$ are the
horizontal coordinates of each characteristic line at the surface,
$\zeta$ is the arclength parameter along the characteristics with
$\zeta=0$ being at the surface and $\zeta = -s^{-1}$ being at the
bottom, and
\begin{equation}
    \sA = \begin{pmatrix}
     1 & 0 & 0 \\ 0 & 1 & c \\
     0 & 0 & s
    \end{pmatrix} \,.
  \label{e.Ab}
\end{equation}
We note that $\det \sA = s$.  Then, for any scalar function $f$,
\begin{equation}
  \label{e.CoV-comp}
  \partial_i (f \circ \Chi)
  = \begin{cases}
        (\partial_i f) \circ \Chi & \text{for } i = x,y \\
        (\bOmega \cdot \grad f) \circ \Chi 
        & \text{for } i=\zeta 
    \end{cases}
\end{equation}
so that, for arbitrary vector fields $\bu$, $\bv$,
\begin{subequations}
\begin{gather}
  \grad \cdot (\bv \circ \Chi)
  = (\grad \cdot \sA \bv) \circ \Chi \,,
  \label{e.div-commute} \\
  (\grad \cdot \sP \bu) \circ \Chi
  = \grad \cdot (\sS \sP \bu \circ \Chi) \,,
  \label{e.div-commute-2}
\end{gather}
\end{subequations}
where \eqref{e.div-commute-2} follows from \eqref{e.div-commute} with
$\bv = \sA^{-1} \sP \bu$ since $\sA^{-1} \sP = \sS \sP$ with
\begin{equation}
  \sS = \operatorname{diag} (1, s^{-2}, s^{-1}) \,.
  \label{e.S-def}
\end{equation}
For any function $\phi$, we define $\bar \phi$ as its mean along the
axis of rotation, i.e.,
\begin{equation}
\label{e.u_mean}
  \bar{\phi} \circ \Chi
  = s \int_{-s^{-1}}^0 \phi \circ \Chi \, d \zeta \,.
\end{equation}
Further, we
write $\hat{\phi} = \phi - \bar{\phi}$ to denote the deviation from the
mean.  The definition for vector fields is analogous.  Since, for
arbitrary functions $\phi$ and $\psi$, $\bar \psi \circ \Chi$ is
independent of $\zeta$ and $\det \sA = s$, we see that
\begin{equation}
  \label{e.barhat}
  \int_\cD \hat{\phi} \, \bar{\psi} \, d\bx
  = s \int_{\bbT^2} \int_{-s^{-1}}^0 \hat{\phi} \circ \Chi \, d \zeta \,
    \bar{\psi} \circ \Chi \, dx
  = 0 \,.
\end{equation}
In other words, mean and fluctuating components in the sense of
\eqref{e.u_mean} are $L^2$-orthogonal.

\section{Thermal wind}
\label{s.twr}

The formal leading order balance in the Euler--Boussinesq momentum
equation \eqref{e.non-boussinesq.a} gives an expression for the
\emph{thermal} or \emph{geostrophic wind} $\bu_g$,
\begin{equation}
  \bOmega \times \bu_g = - \grad p - \rho \, \bk \,.
\end{equation}
Taking the curl to remove the pressure, we obtain the \emph{thermal
wind relation}
\begin{equation}
  \label{eqn:twr}
  \bOmega \cdot \grad \bu_g
  = \begin{pmatrix} - \nabla^\bot \rho \\ 0 \end{pmatrix} \,.
\end{equation}
The characteristics of this first order equation are the lines
parallel to the axis of rotation.  With the notation set up in
Section~\ref{s.mapping-decomp}, it is easy to integrate
\eqref{eqn:twr} along its characteristic lines.  Splitting
$\bu_g = \hat \bu_g + \bar \bu_g$, we first note that \eqref{eqn:twr}
determines only the mean-free component $\hat \bu_g$.  Indeed,
\begin{gather}
  \label{e.twr-mean}
  \partial_\zeta (\hat \bu_g \circ \Chi)
  = (\bOmega \cdot \grad \hat \bu_g ) \circ \Chi
  = \begin{pmatrix}
      - \nabla^\bot \rho \circ \Chi \\ 0
    \end{pmatrix} \,.
\end{gather}
The mean-free component $\hat \bu_g$ must further satisfy
\begin{equation}
  0 = s \int_{-s^{-1}}^0 \hat \bu_g \circ \Chi \, d \zeta
    = \hat \bu_g \circ \Chi(-s^{-1})
      - s \int_{-s^{-1}}^0 \zeta \,
        \partial_{\zeta} (\hat \bu_g \circ \Chi) \, d\zeta \,,
\end{equation}
where we write $\Chi'$ to abbreviate $\Chi(x, y, \zeta')$.  Then,
\begin{align}
  \hat \bu_g \circ \Chi
  & = \hat  \bu_g \circ \Chi(-s^{-1})
      + \int_{-s^{-1}}^\zeta
        \partial_{\zeta} (\hat \bu_g \circ \Chi') \, d\zeta'
      \notag \\
  & = s \int_{-s^{-1}}^0 \zeta \,
        \partial_{\zeta} (\hat \bu_g \circ \Chi) \, d\zeta
      + \int_{-s^{-1}}^\zeta
        \partial_{\zeta} (\hat \bu_g \circ \Chi') \, d\zeta' \,.
  \label{eqn:g}
\end{align}
Inserting the thermal wind relation \eqref{e.twr-mean}, we find that
the vertical component of the thermal wind is independent of $\zeta$,
so $\bk \cdot \hat \bu_g = 0$, and the horizontal components of the
thermal wind are given by
\begin{subequations}
\label{e.ug}
\begin{equation}
  \hat u_g = \nabla^\bot \theta 
\end{equation}
with
\begin{equation}
    \theta \circ \Chi
    = - s \int_{-s^{-1}}^0 \zeta \, \rho \circ \Chi \, d\zeta
      - \int_{-s^{-1}}^\zeta \, \rho \circ \Chi' \, d\zeta' \,.
  \label{e.Theta}
\end{equation}
\end{subequations}
Equation \eqref{e.ug} shows that $\hat u_g$ is horizontally
divergence-free.  Since $\bk \cdot \hat \bu_g = 0$, we also have
$\grad \cdot \bu_g = 0$.

\section{Derivation of the first-order balance model}
\label{s.l1}

We now implement the procedure outlined in
Section~\ref{s.v-asymptotics} at first order in $\eps$.  Using the Lin
constraints for the $\eps$-flow, \eqref{e.lin-constraint} and
\eqref{e.lin-constraint-2}, we can write
\begin{equation}
  \bu_\eps = \bu + \eps \, (\dot{\bv} + \bu \cdot\grad \bv
    - \bv \cdot \grad \bu) + O(\eps^2)
\end{equation}
and 
\begin{equation}
  \rho_\eps = \rho
  - \eps \, \grad \cdot (\rho \bv) + O(\eps^2) \,.
\end{equation}
The transformation vector field $\bv$ will be specified in the
following.  By construction, we assume that all flows are
volume-preserving and leave the domain invariant, so that
$\bu, \bv \in \Vdiv$. For technical
reasons, we also assume that the domain-mean of $\bu$ is zero.  This
is not a restriction since the assumption is only removing a steady
solid-body translation from the system, which is a constant of the
motion of the full Euler--Boussinesq system so that it remains zero if
it vanishes initially.  The transformation vector field $\bv$ shall
also be chosen so as not to generate a solid body translation in balance
model coordinates.

Inserting these relations into the Euler--Boussinesq Lagrangian
\eqref{eqn:lagrangian}, expanded in powers of $\eps$ and truncated to
$O(\eps)$, we obtain after a short computation
\begin{align}
    L & = \int_\cD \bR \circ \Eta_\eps \cdot
    \dot \Eta_\eps + \tfrac{\eps}{2} \, \lvert \dot \Eta_\eps \lvert^2
    \, d \ba - \int_\cD \rho_\eps \, z \, d \bx
    \notag \\
    & = \int_\cD \bu \cdot \sJ \bx - \rho \, z \, d\bx
        + \eps \int_\cD \, \bu \cdot \sJ \bv
        + \tfrac{1}{2} \, \lvert \bu \rvert^2 - \rho \, \bv \cdot \bk \, d \bx
    \notag \\
    & = L_0 + \eps \, L_1 \,.
    \label{e.l-expansion}
\end{align}
Here we have used the divergence theorem to rewrite the potential
energy contribution to $L_1$, where the boundary integral vanishes due
to the boundary condition $\bk \cdot \bv = 0$.

We must now choose the transformation vector field $\bv$ such that it
removes, to the extent possible, all terms that are quadratic in
components of $\bu$.  In preparation, we decompose the kinetic energy
part of $L_1$ as
\begin{equation}
  \int_\cD \lvert \bu \lvert^2 \, d\bx
  = \int_\cD \lvert \hat \bu \rvert^2
        + 2 \, \hat \bu \cdot \bar \bu
        + \lvert \bar \bu \rvert^2 \, d \bx \,.
\label{e.kinetic}
\end{equation}
The cross term in \eqref{e.kinetic} vanishes due to \eqref{e.barhat}.
The square terms are decomposed into terms that can be written as an
$L^2$-pairing with $\sP \hat \bu$, and a final remainder term which
cannot.  Starting with the contribution from
$\lvert \hat \bu \rvert^2$, we write
\begin{equation}
  \int_\cD \lvert \hat \bu \lvert^2 \, d\bx
  = \int_\cD \hat \bu \cdot \sP \hat \bu
      + \hat \bu \cdot \sQ \hat \bu \, d\bx
  = \int_\cD \hat \bu \cdot
      \sP (\hat \bu + \bC[\hat \bu]) \, d\bx \,,
\end{equation}
using Lemma~\ref{l.w-Qu} in the final equality.  The contribution from
$\lvert \bar \bu \rvert^2$ is split differently.  Setting
$\sS_0 = \operatorname{diag} (1, s^{-2}, 0)$, noting that
$s \, (\sI - \sS_0 \sP) \bar \bu = \bOmega \, \bar u_3$, and noting
that $(\sS_0\sP)^\sT-\sS_0\sP$ is skew so that $\sS_0\sP \bar \bu
\cdot \bar \bu - \bar \bu \cdot \sS_0 \sP \bar \bu =0$, we write
\begin{align}
  \int_\cD |\bar \bu|^2 \, d\bx
  & = \int_\cD (\sI + \sS_0 \sP) \bar \bu \cdot
        (\sI - \sS_0 \sP) \bar \bu
      + \sS_0 \sP \bar \bu \cdot \sS_0 \sP \bar \bu \, d\bx
      \notag \\
  & = - s^{-1} \int_\cD \hat \bu \cdot \sP
        \grad (\bOmega \cdot (\sI + \sS_0 \sP) \bar \bu) \,
        z \, d \bx
      + \int_\cD |\sS_0 \sP \bar \bu|^2 \, d\bx \,.
  \label{e.l1-kinenergy}
\end{align}
The second equality is based on Lemma~\ref{l.phi-bar-u3} with
$\bar \phi = s^{-1} \, \bOmega \cdot (\sI + \sS_0 \sP) \bar \bu$.
Note that Lemma~\ref{l.phi-bar-u3} could be used in different ways so
long as we retain a pairing of some function $\bar \phi$ with
$\bar u_3$.  Our chosen splitting is distinguished, due to
Lemma~\ref{l.mean-divergence1}, by the fact that the remainder
integral in \eqref{e.l1-kinenergy} is proportional to the kinetic
energy of a \emph{divergence-free} two-dimensional vector field.  In
the following, this remainder term will be the only term that yields
an evolution equation in the variational principle.  As
$\sS_0 \sP \bar \bu$ is divergence-free, this evolution equation can
always be written in terms of a scalar stream function, i.e., is
determined by the evolution of a single scalar field.  Any other
splitting would yield a two-component evolution equation except for
one special tilt of the axis of rotation, which is not desirable.

Collecting terms, we find
\begin{align}
    \int_\cD \lvert \bu \lvert^2 \, d\bx
    = \int_\cD \hat \bu \cdot \sP \hat \bV \, d \bx
      + \int_\cD \lvert \sS_0 \sP \bar \bu \lvert^2 \, d \bx
  \label{e.KE}
\end{align}
with 
\begin{equation}
  \label{e.V}
  \bV = \hat \bu
    + \bC[\hat \bu]
    - s^{-1} \, \grad \bigl( \bOmega \cdot
      (\sI + \sS_0 \sP) \bar \bu \bigr) \, z \,.
\end{equation}
Even though $\bV$ is not necessarily mean-free, only its mean-free
component $\hat \bV$ contributes to the Lagrangian due to the pairing
with $\hat \bu$.

Since $\sJ$ has a one-dimensional kernel, it is impossible to remove
all quadratic terms from the $L_1$ Lagrangian, but choosing the
transformation vector field $\bv$ as a solution to the equation
\begin{equation}
  \label{e.v}
  \sJ \bv = - \frac{1}{2} \sP \hat \bV \,,
\end{equation}
up to terms that only depend on $\rho$, we can make $L_1$ affine in
$\hat \bu$: Noting that $\sJ$ is invertible on $\Range \sP$ with
pseudo-inverse $\sJ^\sT$, and further that
$\sP \sJ^\sT = \sJ^\sT = \sJ^\sT \sP$, we seek $\bv$ in the form
\begin{equation}
  \bv = \sP \hat \bv + \sQ \hat \bv + \bar \bv
  \label{e.bv}
\end{equation}
with
\begin{equation}
\label{e.Pv-hat}
  \sP \hat \bv
  = - \frac{1}{2} \sJ^\sT \, \hat \bV + \lambda \, \sJ^\sT \hat \bV_g \,.
\end{equation}
In this expression, $\lambda$ is a free parameter and
\begin{equation}
  \bV_g = \hat \bu_g + \bC[\hat \bu_g]
  \label{e.Ug}
\end{equation}
which takes the same form as $\bV$ with $\bu$ replaced by
$\bu_g$. Since $\bu_g$ is not constrained by the thermal wind
relation, we do not include a term that matches the third term of
\eqref{e.V} into the ansatz for $\bV_g$.

The terms $\sQ \hat \bv$ and $\bar \bv$ in \eqref{e.bv} are chosen
such that $\bv$ becomes divergence free and tangent to the top and
bottom boundaries.  Using the construction from
Lemma~\ref{l.constraint}, we write
$\sQ \hat \bv \equiv \bOmega \, \hat g$, where
\begin{gather}
  \hat g \circ \Chi 
  = \hat g \circ \Chi (-s^{-1}) - \int_{-s^{-1}}^\zeta \grad \cdot
      (\sS \sP \hat \bv \circ \Chi') \, d \zeta' \,.
  \label{e.g2}     
\end{gather}
By Lemma~\ref{l.div-uhat}, $\hat \bv$ is divergence-free.  By
Lemma~\ref{l.div-u-bar}, $\bar \bv$ can be chosen such that the zero
flux boundary condition
\begin{align}
  \bar v_3 \circ \Chi
  & = - \bk \cdot \sP \hat \bv \circ \Chi (-s^{-1})
      - s \, \hat g \circ \Chi (-s^{-1})
      \notag \\
  & = - \bk \cdot \sP \hat \bv \circ \Chi (0)
      - s \, \hat g \circ \Chi (0) \,.
\label{e.vn3-bar}
\end{align}
are satisfied and, moreover, $\bar \bv$ is divergence free.
Lemma~\ref{l.uJv} shows that the choice of the horizontal components
$\bar v$ will not enter the computation of the equations of motion.
Only the contribution from $\bar v_3$ will appear, but it is directly
a function of $\hat v_3$ by \eqref{e.vn3-bar}, so that the final
expression will not contain any references to $\bar \bv$.

Combining the contributions from rotation and kinetic energy, we find
\begin{align}
\label{e.l1-KE}
  \int_\cD \bu \cdot \sJ \bv
    + \tfrac{1}{2} \, \lvert \bu \lvert^2 \, d\bx
  & = \int_\cD \hat \bu \cdot \sJ \sJ^\sT \,
        \bigl( - \tfrac{1}{2} \, \hat \bV + \lambda \hat \bV_g\bigr)
        + \tfrac{1}{2} \, \hat \bu \cdot \sP \hat \bV
        + \tfrac{1}{2} \, \lvert \sS_0 \sP \bar \bu \lvert^2 \, d\bx
    \notag \\
  & = \int_\cD
        \lambda \, \hat \bu \cdot \sP \hat \bV_g
        + \tfrac{1}{2} \, \lvert \sS_0 \sP \bar \bu \lvert^2 \, d\bx
    \notag \\
  & = \int_\cD
        \lambda \, \hat u \cdot \hat u_g
        + \tfrac{1}{2} \, \lvert \sS_0 \sP \bar \bu \lvert^2 \, d\bx 
\end{align}
where, in the last equality, we have used Lemma~\ref{l.w-Qu} and the
fact that the geostrophic velocity is exclusively horizontal.  The
contribution to the $L_1$-Lagrangian from the potential energy term is
\begin{equation}
\label{e.l1-PE}
  - \int_\cD \rho \, \bv \cdot \bk \, d\bx
  = \int_\cD \hat u_g \cdot 
      \bigl( \tfrac{1}{2} \, \hat u - \lambda \, \hat u_g \bigr) \, d\bx \,;
\end{equation}
the details of this calculation are given in
Appendix~\ref{a.pe-derivation}.  Altogether, the $L_1$-Lagrangian then
reads
\begin{align}
  L_1
  & = \int_\cD \lambda \, \hat u \cdot \hat u_g
      + \tfrac{1}{2} \, \lvert \sS_0 \sP \bar \bu \lvert^2
      + \hat u_g \cdot
        (\tfrac{1}{2} \, \hat u - \lambda \, \hat u_g) \, d\bx
    \notag \\
  & = \int_\cD \nu \, \hat u \cdot \hat u_g
        + \tfrac12 \, \bar \bu \cdot \sM \bar \bu
        - \lambda \, \lvert \hat u_g \rvert^2 \, d \bx \,,
  \label{e.l1}
\end{align}
where $\nu = \lambda + \tfrac12$ and
\begin{equation}
    \sM = \sP \sS_0^2 \sP = \begin{pmatrix} 1 & 0 & 0 \\
    0 & 1 & - c/s \\
    0 & - c/s & c^2/s^2
    \end{pmatrix} \,.
\end{equation}
In the following, we choose $\lambda$ such that $\nu>0$.

\section{Derivation of the balance model equations of motion}
\label{s.l1-variation}

Taking the variation of \eqref{e.l1}, we obtain
\begin{align}
    \delta L_1
    & = \int_\cD \delta \bu \cdot (\nu \, \hat \bu_g + \sM \bar \bu)
          + \nu \, u \cdot \delta \hat u_g
          - 2 \lambda \, \delta \hat u_g \cdot \hat u_g  \, d\bx
    \notag \\
    & = \int_\cD \delta \bu \cdot \bp
        + \delta \hat u_g \cdot \hat b \, d \bx \,,
\label{e.variationL1}
\end{align}
where
\begin{equation}
   \bp = \sM \bar \bu + \nu \, \hat \bu_g
   \qquad \text{and} \qquad
   \hat b = \nu \, \hat u - 2\lambda \, \hat u_g \,,
\label{e.p-b}
\end{equation}
To rewrite the second term on the right of \eqref{e.variationL1}, we
insert the expression for the geostrophic velocity \eqref{e.ug},
change variables, and recall that horizontal derivatives and
composition with $\Chi$ commute, so that
\begin{align}
  \int_\cD \delta \hat u_g \cdot \hat b \, d\bx
  & = - s \int_\cD \int_{-s^{-1}}^\zeta
          \nabla^\bot \delta \rho \circ \Chi' \, d\zeta' \cdot
          \hat b \circ \Chi \, d\bxi
    \notag \\
  & = s \int_\cD \nabla^\bot \delta \rho \circ \Chi
          \cdot \int_{-s^{-1}}^\zeta \hat b \circ \Chi' \, d\zeta' \, d\bxi
    \notag \\
  & = - s \int_\cD \delta \rho \circ \Chi \,
          \nabla^\bot \cdot \int_{-s^{-1}}^\zeta 
          \hat b \circ \Chi' \, d\zeta' \, d\bxi
    \notag \\
  & = - \int_\cD \delta \rho \, \nabla^\bot \cdot B \, d\bx \,,
  \label{e.delta-ug}
\end{align}
where $B$ is the anti-derivative of $b$ along the axis of
rotation, i.e.,
\begin{equation}
\label{e.B}
  B = \nu \, U - 2\lambda \, U_g 
\end{equation}
where
\begin{equation}
  U \circ \Chi
  = \int_{-s^{-1}}^\zeta \hat u \circ \Chi' \, d\zeta'
 \label{e.bU}
\end{equation}
and $U_g$ is defined likewise.  Inserting \eqref{e.delta-ug} into
\eqref{e.variationL1}, we have
\begin{equation}
  \delta L_1 = \int_\cD \delta \bu \cdot \bp
  - \delta \rho \, \nabla^\bot \cdot B \, d \bx \,.
\end{equation}
We recall the general Euler--Poincar\'e equation
\eqref{eqn:variational3}, which can be written
\begin{equation}
  \partial_t \bm + (\grad \times \bm) \times \bu
  + \frac{\delta \ell}{\delta \rho} \, \grad \rho = \grad \tilde \phi
  \label{e.general-ep}
\end{equation}
with $\tilde \phi$ an arbitrary scalar field.  Here,
\begin{equation}
\label{e.var}
  \bm = \frac{\delta \ell}{\delta \bu} = \bR + \eps \, \bp
  \qquad \text{and} \qquad
  \frac{\delta \ell}{\delta \rho} = - z - \eps \, \nabla^\bot \cdot B \,,
\end{equation}
so that, with $\tilde \phi = -\phi - z \rho$,
\begin{equation}
  \label{e.euler-poincare}
  \bOmega \times \bu + \rho \bk
  + \eps \, (\partial_t \bp + (\grad \times \bp) \times \bu
  - \grad \rho \, \nabla^\bot \cdot B) = - \grad \phi \,.
\end{equation}
As $c \, \partial_z = - s \, \partial_y$ when applied to averaged
quantities, we find that
$\grad \times \sM \bar \bu = \bOmega \, \omega$ with
$\omega = s^{-1} \, \nabla^\bot \cdot \sM_h \bar \bu$.  Hence,
\begin{equation}
  \bxi
  \equiv \grad \times \bp
  = \bOmega \, \omega + \nu \, \bgamma 
  \label{e.xi}
\end{equation}
with
\begin{equation}
  \bgamma 
  = \grad \times \hat \bu_g
  = - \grad \times \curl \theta
  = \begin{pmatrix}
      - \nabla \partial_z \theta \\
      \Delta \theta
    \end{pmatrix} \,,
  \label{e.xi2}
\end{equation}
where we recall that $\hat u_g = \nabla^\bot \theta$, write $\curl f = \grad \times (0,0,f)$ to identify
the curl of a scalar field with the curl of a vector field oriented in
the vertical, and use $\Delta$ to denote the \emph{horizontal}
Laplacian.  With this notation in place, taking the curl of
\eqref{e.euler-poincare} and noting that
$\grad \times (\bxi \times \bu ) = \bu \cdot \grad \bxi - \bxi \cdot
\grad \bu$ as both $\bu$ and $\bxi$ are divergence-free, we find
\begin{equation}
\label{e.ep-cross}
  - \bOmega \cdot \grad \bu
  + \grad \times (\rho \bk)
  + \eps \, \bigl(
      \partial_t \bxi
      + \bu \cdot \grad \bxi
      - \bxi \cdot \grad \bu
      + \grad \rho \times \grad \nabla^\bot \cdot B
    \bigr) = 0 \,.
\end{equation}
The corresponding balance model Hamiltonian, via
\eqref{e.hamiltonian}, is given by
\begin{equation}
  H = \int_\cD \rho \, z
      + \eps \, \bigl( \tfrac{1}{2} \, \bar \bu \cdot \sM
        \bar \bu + \lambda \, |\hat \bu_g|^2 \bigr) \, d \bx \,. 
\end{equation}
The potential vorticity for the balance model reads
\begin{align}
\label{e.potential_vorticity}
  q & = (\grad \times \bm) \cdot \grad \rho
    \notag \\
    & = (\bOmega + \eps \, \bOmega \, \omega + \eps \nu \, \bgamma)
        \cdot \grad \rho
    \notag \\
    & = (s + \eps s \, \omega + \eps \nu \, \Delta \theta) \, \partial_z \rho
        + (\Omega + \eps \, \Omega \, \omega -
           \eps \nu \, \partial_z \nabla \theta) \cdot \nabla \rho
    \notag \\
    & = - (s + \eps s \, \omega + \eps \nu \, \Delta \theta) \,
          \partial_z (\bOmega \cdot \grad \theta)
        - (\Omega + \eps \, \Omega \, \omega -
           \eps \nu \, \partial_z \nabla \theta)
           \cdot \nabla (\bOmega \cdot \grad \theta) \,,
\end{align}
where $\Omega=(0,c)$. In the last equality, we have used that
$\rho=-\bOmega \cdot \grad \theta$.  Alternatively, we can write
\begin{equation}
    q = \begin{vmatrix}
          - (s + \eps s \, \omega + \eps \nu \, \Delta \theta) &
          \nabla (\bOmega \cdot \grad \theta) \\
          (\Omega + \eps \, \Omega \, \omega - \eps \nu \,
            \partial_z \nabla \theta) &
          \partial_z (\bOmega \cdot \grad \theta) 
    \end{vmatrix} \,.
\end{equation}
We remark that the relation between $q$ and $\theta$ can be seen as a
second order nonlinear differential operator of the form 
\begin{equation}
    q = F (\omega, D^2 \theta) \,,
\end{equation}
which is nonlinearly elliptic \citep[e.g.][]{gilbarg} so long as
\begin{equation}
    [F_{ij}] = \frac{\partial F}{\partial (D^2 \theta)}
\end{equation}
is positive (or negative) definite.  Direct computation shows that
\begin{equation}
  \label{e.pv-matrix}
  [F_{ij}] = 
  \begin{pmatrix}
    \eps \nu \partial_z \rho
    & - \tfrac{1}{2} \eps c \, \xi_1
    & - \tfrac{1}{2} \eps \, ( \xi_1 + \nu \partial_x \rho) \\
    - \tfrac{1}{2} \eps c \, \xi_1
    & \eps \nu \partial_z \rho
      - c( c + \eps \xi_2)
    & - \tfrac{1}{2} \bigl( 2cs + \eps
      ( c \xi_3 + s \xi_2 + \nu \partial_y \rho) \bigr) \\
    - \tfrac{1}{2} \eps \, ( \xi_1 + \nu \partial_x \rho)
    & - \tfrac{1}{2} \bigl( 2cs + \eps
      ( c \xi_3 + s \xi_2 + \nu \partial_y \rho) \bigr)
    & - s^2 - \eps s \, \xi_3
  \end{pmatrix} \,,
\end{equation}
so that $-F$ is positive definite provided its principal minors are
positive, i.e.,
\begin{subequations}
\label{e.ellp-cond}
\begin{gather}
  - \eps \nu \partial_z \rho  > 0 \,,
  \label{e.cond-a} \\
  - \eps \nu \, \partial_z \rho + c \bigl[ c + \eps \bigl( \xi_2 + \tfrac{1}{4} \tfrac{c \xi_1^2}{\nu \partial_z \rho} \bigr) \bigr] > 0 \,,
  \label{e.cond-b} \\
  \det F < 0 \,.
  \label{e.cond-c}
\end{gather}
\end{subequations}
Condition \eqref{e.cond-c} can be written more explicitly as
\begin{equation}
   - (\eps s \nu \partial_z \rho)^2 + \eps^2 \, f(\bxi, \nabla \rho;
   \bxi, \rho) < 0 \,,
\end{equation}
where $f$ is linear in its first two arguments.  Hence, these
conditions are satisfied if the fluid is stably stratified so that
$\partial_z \rho < 0$, the deviations from a steady mean state are
sufficiently small, and $\eps$ is sufficiently small.

\section{Separation of balance relation into dynamic and kinematic
components}
\label{s.bal-comp}

In the following, we show that the balance relation is mostly a
kinematic relationship between density $\rho$ and balanced velocity
field.  However, there is one horizontal scalar field, $\omega$, which
evolves dynamically via the vertical component equation
\begin{equation}
  \label{e.ep-cross-vertical}
  \partial_t \xi_3
  + \bu \cdot \grad \xi_3
  = \bxi \cdot \grad u_3
  + \nabla \rho \cdot \nabla^\bot \nabla^\bot \cdot B
  + \eps^{-1} \, \bOmega \cdot \grad \hat u_3 \,.
\end{equation}
Note that the seemingly unbalanced $O(\eps^{-1})$-term in this
equation is actually only an $O(1)$-contribution because $\bu$ is an
$O(\eps)$ perturbation of $\bu_g$ whose vertical component is zero.

Taking the average of \eqref{e.ep-cross-vertical} along the axis of
rotation and using Lemma~\ref{l.av-commutation-1} and
Lemma~\ref{l.z-commute-2} from the appendix to commute averaging with
directional derivatives where possible, we find that the evolution
equation for $\xi_3$ reduces to a prognostic equation for $\omega$,
\begin{equation}
\label{e.average-vertical}
  \partial_t \omega + \bar \bu \cdot \grad \omega
   = \overline{\nabla \rho \cdot \nabla^\bot \nabla^\bot \cdot B}
  - \nu \, \overline{(\partial_z \nabla \theta \cdot \nabla u_3
    - \Delta \theta \, \partial_z u_3
    + \bu \cdot \grad \Delta \theta)} \,.
\end{equation}
All other contributions to \eqref{e.ep-cross} are entirely kinematic.
To see this, we start with multiplying \eqref{e.average-vertical} by
$\bOmega$, then subtract this expression from \eqref{e.ep-cross} to
obtain
\begin{align}
\label{e.bal-1}
  0 = & - \bOmega \cdot \grad \bu
  + \grad \times (\rho \bk)
  + \eps \, \bigl[
      \nu \, \partial_t \bgamma
      + \bOmega \, \hat \bu \cdot \grad \omega
      + \nu \, \bu \cdot \grad \bgamma
      - \bxi \cdot \grad \bu
    \notag \\
    & + \grad \rho \times \grad \nabla^\bot \cdot B
    + \bOmega \,
      \overline{\nabla \rho \cdot \nabla^\bot \nabla^\bot \cdot B}
      + \nu \, \bOmega \, \overline{(\bgamma \cdot \grad u_3
    - \bu \cdot \grad \Delta \theta)} \bigr] \,.
\end{align}
We now assume that the flow is strongly and uniformly
stratified.  More specifically, we shall assume that
\begin{equation}
  \partial_z \rho = - \alpha + \partial_z \tilde \rho
  \label{e.stratification}
\end{equation}
where $\tilde \rho$ is small in a suitable norm, to be specified
further below.  It is possible to weaken this assumption and allow
variations in the stratification profile so long as stratification is
uniformly stable across the layer, but the simpler assumption
\eqref{e.stratification} makes the structure of the problem more
transparent.

Inserting \eqref{e.stratification} into the definition of $\bgamma$
and noting that $\bOmega \cdot \grad \theta = - \rho$, we find
\begin{equation}
  \bOmega \cdot \grad \dot \gamma_3
  = \Delta (\bOmega \cdot \grad \dot \theta)
  = \Delta (\bu \cdot \grad \rho)
  = - \alpha \, \Delta u_3 + \Delta (\bu \cdot \grad \tilde \rho) \,.
  \label{e.id1}
\end{equation}
Decomposing
\begin{equation}
  U = \nabla^\bot \Psi + \nabla \Phi \,,
  \label{e.Udecomp}
\end{equation}
and, setting $\zeta = \Delta \Psi$, we have
\begin{equation}
  \grad \rho \times \grad \nabla^\bot \cdot B
  = - \alpha \nu \, \bk \times \grad \zeta
    + \nu \, \grad \tilde \rho \times \grad \zeta
    - 2 \lambda \, \grad \rho \times \grad \nabla^\bot \cdot U_g \,.
  \label{e.id2}
\end{equation}
Inserting \eqref{e.id1} and \eqref{e.id2} into \eqref{e.bal-1}, taking
the horizontal curl of the horizontal component equations, and
applying $\bOmega \cdot \grad$ to the vertical component equation, we
obtain a pair of kinematic balance relations:
\begin{subequations}
\label{e.balance}
\begin{align}
\label{e.bal-lhs}
  (\bOmega \cdot \grad)^2 \zeta + \eps \alpha \nu \, \Delta \phi \zeta
  & = - \Delta \tilde \rho + \eps \, f \,, \\
  (\bOmega \cdot \grad)^2 u_3 + \eps \alpha \nu \, \Delta \phi u_3
  & = \eps \, g
\end{align}
with
\begin{align}
  f
  & = \Omega \cdot \nabla^\bot \, (\hat \bu \cdot \grad \omega)
        + \nu \, \nabla^\bot \cdot (\bu \cdot \grad \gamma)
        - \nabla^\bot \cdot (\bxi \cdot \grad u)
        - \nu \, \nabla \cdot
            (\nabla \tilde \rho \, \partial_z \zeta)
    \notag \\
  & \quad
        + \nu \, \nabla \cdot
            (\nabla \zeta \, \partial_z \tilde \rho)
        + 2 \lambda \, \nabla \cdot
            (\nabla \rho \, \partial_z \nabla^\bot \cdot U_g)
        - 2 \lambda \, \nabla \cdot
            (\nabla \nabla^\bot \cdot U_g \, \partial_z \rho)
    \notag \\
  & \quad
      - 2 \lambda \, \Omega \cdot \nabla^\bot \,
        \overline{\nabla \tilde \rho \cdot \nabla^\bot \nabla^\bot \cdot U_g}
        + \nu \, \Omega \cdot \nabla^\bot \,
      \overline{(\nabla \tilde \rho \cdot \nabla^\bot \zeta
        + \bgamma \cdot \grad u_3
        - \hat \bu \cdot \grad \Delta \theta)} \,, \\
  g & = \nu \, \Delta ( \bu \cdot \grad \tilde \rho)
        + s \, \hat \bu \cdot \grad \omega
        + \nu \, \bOmega \cdot \grad (\bu \cdot \grad \gamma_3
          - \bxi \cdot \grad u_3
          - \nabla \tilde \rho \cdot \nabla^\bot \nabla^\bot \cdot B) \,.
\end{align}
\end{subequations}
This elliptic problem is augmented by homogeneous Dirichlet boundary
conditions on both $\zeta$ and $u_3$.  These boundary conditions
encode, for $u_3$, the no-flux conditions of the full velocity field,
and for $\zeta$ that it is the antiderivative along $\bOmega$ of a
mean-free field with an arbitrary choice of gauge that disappears upon
differentiation.

The right hand functions $f$ and $g$ still contain terms that depend
on the unknown functions $\zeta$ and $u_3$, so they need to be solved
by fixed point iteration.  In the next section, we will argue that
this can be done under suitable smallness assumptions for
$\tilde \rho$ and $\omega$.

\section{Closing the balance model}
\label{s.solution}

To close the balance relation and the prognostic equations, we need to
recover the full velocity field $\bu$ from $\omega$, $\zeta$, and
$u_3$.  We express the horizontal component of the velocity in terms
of stream function $\psi$ and velocity potential $\phi$,
\begin{equation}
    u = \nabla^\bot \psi + \nabla \phi \,.
\end{equation}
First, note that
$\nabla^\bot \cdot \sM_h \bar \bu = \nabla^\bot \cdot \bar u - c/s \,
\partial_x \bar u_3 = s \, \omega$ by definition, so that
\begin{subequations}
  \label{e.reconstruction}
\begin{equation}
  \Delta \bar \psi
  = s \, \omega + \tfrac{c}{s} \, \partial_x \bar u_3 \,.
\end{equation}
Next, due to \eqref{e.Udecomp},
\begin{equation}
  \label{e.hat-psi}
  \Delta \Psi = \zeta \,.
\end{equation}
Finally, by incompressibility,
\begin{equation}
  \Delta \phi = - \partial_z u_3 
\end{equation}
so that, altogether,
\begin{equation}
  u = \nabla^\bot \bar \psi
      + \bOmega \cdot \grad \nabla^\bot \Psi
      + \nabla \phi \,.
\end{equation}
\end{subequations}
This expression for the horizontal components of the velocity field,
first, proves that the kinematic balance relation \eqref{e.balance}
can be solved by iteration.  Seeking a solution in the Sobolev space
$H^{s+2}$, $H^s$ denoting the space of square integrable functions
with square-integrable derivatives up to order $s$ where we take $s$
large enough so that products of such functions are also contained in
$H^s$, we need to verify that all terms that appear in $f$ and $g$ are
contained in $H^s$.  For example, for the term
$\nabla^\bot \cdot (\bxi \cdot \grad u)$ which appears in the
expression for $f$, the highest derivatives on the unknown functions
appear as
\begin{equation}
  \bxi \cdot \grad(\partial_x \bar u_3
  + \bOmega \cdot \grad \zeta) \,,
\end{equation}
so this term is contained in $H^s$ provided $\zeta, u_3 \in H^{s+2}$.
All other terms are either similar or have only lower-order
derivatives on the unknowns.  Further, the coefficients that appear,
here $\bxi$, depend only on $\omega$ and $\tilde \rho$, so they are
small if the data is close enough to the stably stratified rest state
in a Sobolev norm of sufficiently high order.  Thus, the balance
relation \eqref{e.balance} defines a contraction in $H^{s+2}$ and can
be solved by iteration.

Second, the reconstructed velocity field $\bu$ is used to propagate
$\omega$ and $\rho$ (or, alternatively, the potential vorticity $q$),
in time.  Then the complete set of balance model equations is given by
the the kinematic balance relation \eqref{e.balance}, the
reconstruction equations \eqref{e.reconstruction}, the continuity
equation in three dimensions,
\begin{subequations}
\begin{gather}
  \partial_t \rho + \bu \cdot \grad \rho = 0 \,,
\end{gather}
and the evolution equation in two spatial dimensions for $\omega$,
\begin{equation}
  \partial_t \omega + \bar \bu \cdot \grad \omega
  = \overline{\nabla \rho \cdot \nabla^\bot \nabla^\bot \cdot B}
    - \nu \, \overline{(\partial_z \nabla \theta \cdot \nabla u_3
    - \Delta \theta \, \partial_z u_3
    + \bu \cdot \grad \Delta \theta)} \,.
\end{equation}
\end{subequations}
where $B$ is defined in \eqref{e.B} and $\theta$ is the geostrophic
stream function which depends on $\rho$ via \eqref{e.Theta}.

\section{Discussion and Conclusion}
\label{s.conclusion}

In this paper, we have achieved a variational model reduction for the
full three-dimensional Euler--Boussinesq equation with a full Coriolis
force.  We have studied a simple setting, namely the $f$-plane
approximation, a layer of fluid of constant depth, and periodic
boundary conditions in the horizontal.  The picture which emerges is
structurally identical to that of variational balance models for the
primitive equations as derived by \cite{Salmon1996} and generalized in
\cite{OliverV:2016:GeneralizedLS}:
\begin{enumerate}[label={\upshape(\roman*)}]
\item There are two prognostic variables of the first-order balance
model, the density $\rho$ (equivalently, a potential temperature) and
a scalar generalized vorticity $\omega$.

\item The generalized vorticity $\omega$ depends only on the two
velocity components that are perpendicular to the axis of rotation,
and it is averaged along the axis of rotation.  Thus, $\omega$ is
independent of the oblique vertical coordinate $\zeta$.  

\item All other components of the velocity field, i.e., all deviations
from the vertical mean as well as the vertical mean of the velocity
component pointing along the axis of rotation, are kinematic.  In
other words, these components are slaved to $\rho$ and $\omega$ via a
balance relation.

\item The balance relation is elliptic if rotation is sufficiently
fast and the prognostic fields are small perturbations of a stably
stratified equilibrium state.

\item The balance model conserves energy and has a materially
conserved potential vorticity.
\end{enumerate}
Thus, we have verified that the variational derivation of balance
models of semigeostrophic type extends all the way to one of the most
general models of geophysical flows.  In particular, the assumption of
hydrostaticity and of the ``traditional approximation'' changes
details, but does not change the structural features of the
semigeostrophic limit.

At the same time, a full Coriolis force causes difficulties that are
not seen in the simpler cases.  Since the axis of rotation is not
aligned with the direction of gravitational force, there are two
distinguished ``vertical'' directions.  This is an obstacle to using a
fully intrinsic geometric formulation of the derivation in the spirit
of \cite{arnold1999topological} or \cite{GilbertV:2018:GeometricGL},
forcing us to resort to detailed coordinate calculations.  The
resulting equations, therefore, appear to lack the relative simplicity
of balance models in more idealized settings; many of the new terms
simplify or disappear when the Coriolis force is acting exactly in the
horizontal plane.

We emphasize that our derivation requires a nontrivial change of
coordinates already at $O(\eps)$.  The associated transformation
vector field depends on the prognostic part of the mean velocity, cf.\
the last term in \eqref{e.V}.  We believe that it is not possible to
remove this contribution to the transformation at $O(\eps)$ since
there is no leading-order constraint through the thermal wind relation
on this component.  For the analogous computation for the primitive
equations, it suffices to set $\lambda = \tfrac12$ in
\eqref{e.Pv-hat}, which formally cancels all terms at $O(\eps)$.  For
the Euler--Boussinesq system, it is the last term in \eqref{e.V} that
cannot be canceled.  This additional contribution to the
transformation vector field appears even in the case when the axis of
rotation is aligned with the geometric vertical and only disappears
when the hydrostatic approximation is made.  Thus, the more
straightforward derivation by \cite{Salmon1996}, who inserts the
thermal wind relation directly into the extended Lagrangian to
constrain the variational principle does not work in this case; the
more general setting described in Section~\ref{s.v-asymptotics} must
be used.

When the axis of rotation is aligned with the horizontal, i.e., for
the equatorial $f$-plane, all our final expressions remain
non-singular, which is surprising given that some of the intermediate
expressions do contain diverging terms.  What fails, however, is
ellipticity of the balance relation.  Stratification is providing
regularization in the horizontal plane, rotation is providing
regularization in the vertical.  When the axis of rotation is tilted
into the horizontal, all control in the geometric vertical is lost.
We do not expect that the model remains well posed in this limit.

The balance relation allows, under the conditions stated, stable
reconstruction of the slaved components of the velocity field from
sufficiently smooth prognostic variables.  A full analysis of
well-posedness of the balance model is more difficult, as the right
hand side of the balance relation contains high-order derivatives of
the prognostic variables, and remains open.  In practical terms, the
balance relation might be most useful as a diagnostic relation
independent of the full system of prognostic relations.  Nonetheless,
the full balance model can be solved numerically in the formulation
given in Section~\ref{s.solution}.  For balance models in two
dimensions, potential vorticity based methods provide an alternative,
numerically robust setting \citep{DritschelGO:2016:ComparisonVB}.  As
for primitive equation balance models, potential vorticity based
numerics require the solution of a nonlinear elliptic equation
\citep[cf.][]{OliverV:2016:GeneralizedLS,AkramovO:2020:ExistenceSB}.
This alternative formulation would require the solution of one more
Monge--Amp\`ere-like equation, here with oblique derivatives, but may
be more stable because the potential vorticity is materially advected.

\section*{Acknowledgment}

This paper is a contribution to project M2 (Systematic Multi-Scale
Analysis and Modeling for Geophysical Flow) of the Collaborative
Research Center TRR 181 ``Energy Transfers in Atmosphere and Ocean''
funded by the Deutsche Forschungsgemeinschaft (DFG, German Research
Foundation) under project number 274762653.  Additional funding was
received through the Ideen- und Risikofund 2020 at Universität
Hamburg.

\section*{Data Availability}
Data sharing is not applicable to this article as no new data were created or analyzed in this study.

\section*{Appendix A. Averaging along the axis of rotation}
\label{a.averaging}
\renewcommand{\theequation}{A.\arabic{equation}}
\setcounter{equation}{0}

This appendix states two simple lemmas on the properties of the
oblique averaging operation.  The first shows that it distributes over
products as expected.

\begin{lemma} \label{l.av-commutation-1}
Let $\phi$ and $\psi$ be arbitrary functions and $\bv$ an arbitrary
vector field on $\cD$.  Then
\begin{enumerate}[label={\upshape(\roman*)}]
\item\label{0.av-commutation}
$\overline{\psi \, \bar \phi} = \bar \psi \, \bar \phi$,
\item\label{i.av-commutation}
$\overline{\bv \cdot \grad \bar \phi} = \bar \bv \cdot \grad \bar
\phi$.
\end{enumerate}
\end{lemma}

\begin{proof}
For \ref{0.av-commutation}, note that $\bar \phi \circ \Chi$ is
independent of $\zeta$, so that
\begin{equation}
  \overline{\psi \, \bar \phi} \circ \Chi
  = s \int_{-s^{-1}}^0 \psi \circ \Chi \,
        \bar \phi \circ \Chi \, d \zeta
      \notag \\
  = s \int_{-s^{-1}}^0 \psi \circ \Chi \, d \zeta \,
        \bar \phi \circ \Chi
  = \bar \psi \circ \Chi \, \bar \phi \circ \Chi \,.
\end{equation}
For \ref{i.av-commutation}, note that
$\bOmega \cdot \grad \bar \phi=0$, so that the $z$-derivative can be
replaced by an equivalent $y$-derivative.  Since horizontal
derivatives commute with taking the average, part
\ref{0.av-commutation} applies and yields the claim.
\end{proof}

Commutation of vertical derivatives of arbitrary functions with
averaging is more subtle, as the next lemma shows.  Here and in the
following, we write $\Chi(0)$ and $\Chi(-s^{-1})$ to indicate that the
expression is evaluated at the top ($\zeta = 0$) or at the bottom
boundary ($\zeta=s^{-1}$), as a function of the remaining horizontal
variables.

\begin{lemma}
\label{l.z-commute-2}
Let $\phi$ be a function with
$\phi \circ \Chi(0) = \phi \circ \Chi(-s^{-1})$ and $\bv$ an arbitrary
vector field on $\cD$.  Then
\begin{enumerate}[label={\upshape(\roman*)}]
\item\label{i.z-commute-2} $\partial_z \bar \phi =
\overline{\partial_z \phi}$,
\item\label{ii.av-commutation}
$\overline{\bar\bv \cdot \grad \phi} = \bar \bv \cdot \grad \bar
\phi$.
\end{enumerate}
\end{lemma}

\begin{proof}
On the one hand, $\bOmega \cdot \grad \bar \phi = 0$, so that
$s \, \partial_z \bar \phi = - c \, \partial_y \bar \phi$.  On the
other hand,
\begin{align}
  \overline{\partial_z \phi} \circ \Chi
  & = s \int_{-s^{-1}}^0 (\partial_z \phi) \circ \Chi \, d \zeta
      \notag \\
  & = \int_{-s^{-1}}^0 (\partial_\zeta - c \, \partial_y)
        (\phi \circ \Chi) \, d \zeta
      \notag \\
  & = - \frac{c}s \, \partial_y \bar \phi \circ \Chi
      + \phi \circ \Chi (0) - \phi \circ \Chi (-s^{-1}) \,.
  \label{e.proof.z-commute-2}
\end{align}
Under the condition stated, the boundary terms cancel.  This implies
\ref{i.z-commute-2}.  For \ref{ii.av-commutation}, we use
Lemma~\ref{l.av-commutation-1} to move $\bar \bv$ out of the average.
Horizontal derivatives can be moved out of the average without
restrictions; for the $z$-derivative, part \ref{i.z-commute-2}
applies.
\end{proof}

\section*{Appendix B. Splitting of divergence free vector fields}
\label{a.divergence}
\renewcommand{\theequation}{B.\arabic{equation}}
\setcounter{equation}{0}

We proceed to prove a number of identities which describe the
splitting of divergence-free vector fields with zero-flux boundary
conditions into, on the one hand, mean and mean-free components and,
on the other hand, components along and perpendicular to the axis of
rotation.

\begin{lemma} \label{l.mean-divergence1}
Let $\bv \in \Vdiv$.  Then $\nabla \cdot \sS_h \sP \bar \bv = 0$.
\end{lemma}

\begin{proof}
For a divergence-free vector field,
$\grad \cdot \sP \bv = - \grad \cdot \sQ \bv = - \bOmega \cdot \grad
(\bOmega \cdot \bv)$, so that \eqref{e.div-commute-2} turns into
\begin{equation}
  \nabla \cdot (\sS_h \sP \bv \circ \Chi)
  = - \partial_\zeta (\bOmega \cdot \bv \circ \Chi)
    - \partial_\zeta (\sS_3 \sP \bv \circ \Chi)
  = - s^{-1} \, \partial_\zeta (v_3 \circ \Chi) \,,
\end{equation}
where the last equality can be verified by direct computation in
coordinates.  Integrating in $\zeta$ and noting that the right hand
side is zero due to the boundary conditions, we obtain the statement
of the lemma.
\end{proof}

The following is a somewhat weaker converse of
Lemma~\ref{l.mean-divergence1}. 

\begin{lemma} \label{l.mean-divergence2}
Let $\bv$ be a vector field with $\nabla \cdot \sS_h \sP \bar \bv =
0$.  Then $\grad \cdot \bar \bv = 0$.
\end{lemma}

\begin{proof}
The assumption implies
\begin{equation}
  0 = \nabla \cdot \bar v
      - \frac{c}s \, \partial_y \bar v_3 \,.
  \label{e.assumption-coords}
\end{equation}
Since $\bar v_3 \circ \Chi$ is independent of $\zeta$, this implies
$\grad \cdot (\sA^{-1} \bar \bv \circ \Chi) =0$.  By
\eqref{e.div-commute}, this implies that $\grad \cdot \bar \bv = 0$.
\end{proof}

\begin{corollary}
If $\bv \in \Vdiv$, then $\grad \cdot \hat \bv = \grad \cdot \bar \bv
= 0$.
\end{corollary}

\begin{proof}
Lemma~\ref{l.mean-divergence1} followed by
Lemma~\ref{l.mean-divergence2} yields $\grad \cdot \bar \bv=0$.  Then
$\hat \bu = \bu - \bar \bu$ is also divergence-free.
\end{proof}

\begin{lemma}
\label{l.constraint}
Let $\bu \in \Vdiv$.  Then $\sQ \hat \bu$ is uniquely determined by
$\sP \hat \bu$ and given by the formula
$\sQ \hat \bu = \bOmega \hat g$ with
\begin{gather}
  \hat g \circ \Chi 
  = - s \int_{-s^{-1}}^0 \zeta \, \grad \cdot
        (\sS \sP \hat \bu \circ \Chi) \, d \zeta
    - \int_{-s^{-1}}^\zeta \grad \cdot
        (\sS \sP \hat \bu \circ \Chi') \, d\zeta'  \,.
  \label{e.g}     
\end{gather}
\end{lemma}

\begin{proof}
With $\bu = \sP \bu + \sQ \bu$, the divergence condition reads
\begin{align}
\label{eqn:q}
  0 & = \grad \cdot\bu = \grad \cdot \sP \bu + \grad \cdot \sQ \bu \,.
\end{align}
Recalling that $\sQ = \bOmega \bOmega^\sT$ and setting
$g= \bOmega \cdot \bu$, we can write $\sQ \bu = \bOmega g$.  Then,
\begin{equation}
  \bOmega \cdot \grad \hat g
  = \bOmega \cdot \grad g
  = - \grad \cdot \sP \bu 
\end{equation}
so that, due to \eqref{e.CoV-comp},
\begin{equation}
  \partial_\zeta (\hat g \circ \Chi)
  = (\bOmega \cdot \grad \hat g) \circ \Chi
  = - (\grad \cdot \sP\bu) \circ \Chi
  = - \grad \cdot (\sS \sP \bu \circ \Chi)
  = - \grad \cdot (\sS \sP \hat \bu \circ \Chi) \,.
  \label{e.dzg}
\end{equation}
The third equality in \eqref{e.dzg} is due to \eqref{e.div-commute-2}
and the last equality is due to $\bu = \hat \bu + \bar \bu$, where
$\bar \bu \circ \Chi$ is independent of $\zeta$, so that the entire
contribution from $\bar \bu$ vanishes by
Lemma~\ref{l.mean-divergence1}.  Equation \eqref{e.dzg} determines $g$
uniquely up to a constant of integration on each of the characteristic
lines; we write $g = \hat g + \bar g$ and choose $\bar g$ as this
constant of integration.  The condition that $\hat g$ is mean-free
along each characteristic line implies
\begin{gather}
\label{e.meanFree}
  0 = s \int_{-s^{-1}}^0 \hat g \circ \Chi \, d \zeta
    = \hat g \circ \Chi (-s^{-1}) - s \int_{-s^{-1}}^0 \zeta \,
        \partial_\zeta (\hat g \circ \Chi) \, d \zeta \,.
\end{gather}
Then, substituting \eqref{e.dzg} and \eqref{e.meanFree} into
\begin{gather}
\label{e.meanFree-g}
  \hat g \circ \Chi (\zeta)
  = \hat g \circ \Chi (-s^{-1})
    + \int_{-s^{-1}}^\zeta
      \partial_{\zeta} (\hat g \circ \Chi') \, d \zeta' \,,
\end{gather}
we obtain \eqref{e.g}, which determines $\sQ \hat \bu$ uniquely.
\end{proof}

The next lemma provides a converse statement to
Lemma~\ref{l.constraint}.

\begin{lemma}
\label{l.div-uhat}
Suppose $\sP \hat \bu$ is a given vector field which is mean-free and
contained in the range of $\sP$.  Define $\hat g$ as in
Lemma~\ref{l.constraint}, i.e.,
\begin{equation}
  \label{e.g1}
  \hat g \circ \Chi
  = - s \int_{-s^{-1}}^0 \zeta \, \grad \cdot
      (\sS \sP \hat \bu \circ \Chi) \, d\zeta
    - \int_{-s^{-1}}^\zeta \grad \cdot
      (\sS \sP \hat \bu \circ \Chi') \, d\zeta' \,.
\end{equation}
Then $\hat \bu = \sP \hat \bu + \sQ \hat \bu$ with
$\sQ \hat \bu = \bOmega \hat g$ is mean-free and divergence-free.
\end{lemma}

\begin{proof}
The fact that $\hat g$, hence $\hat \bu$, is mean-free is a direct
consequence of the choice of constant of integration in the proof of
Lemma~\ref{l.constraint}.

To prove that $\hat \bu$ is divergence-free, we take the
$\zeta$-derivative of \eqref{e.g1},
\begin{equation}
  \partial_\zeta (\hat g \circ \Chi)
  = - \grad \cdot (\sS \sP \hat \bu \circ \Chi) \,.
\end{equation}
This implies that
\begin{align}
  0 & = \partial_\zeta (\hat g \circ \Chi)
          + \grad \cdot (\sS \sP \hat \bu \circ \Chi)
    \notag \\
    & = (\bOmega \cdot \grad \hat g) \circ \Chi
          + (\grad \cdot \sP \hat \bu) \circ \Chi
    \notag \\
    & = (\grad \cdot \sQ \hat \bu) \circ \Chi
          + (\grad \cdot \sP \hat \bu) \circ \Chi
    \notag \\
    & = (\grad \cdot \hat \bu) \circ \Chi \,.
\end{align}
Since $\Chi$ is invertible, we find that $\grad \cdot \hat \bu=0$.
\end{proof}

\section*{Appendix C. Inner product identities for decomposed vector fields}
\label{a.integration}
\renewcommand{\theequation}{C.\arabic{equation}}
\setcounter{equation}{0}

\begin{lemma}
\label{l.w-Qu}
Let $\bu \in \Vdiv$; we write $\hat \bu$ to denote its mean-free
component as before.  Further, let $\hat \bw$ be any mean free vector
field.  Then
\begin{equation}
  \int_\cD \hat \bw \cdot \sQ \hat \bu \, d \bx
  = \int_\cD \hat \bu \cdot \sP \bC[\hat \bw] \, d\bx
\end{equation}
with $\bC[\hat \bw]$ defined by
\begin{equation}
\label{e.Aw}
  \bC[\hat \bw] \circ \Chi
  = - \sS \grad \int_{-s^{-1}}^\zeta
      \bOmega \cdot \hat \bw \circ \Chi' \, d\zeta' \,.
\end{equation}
\end{lemma}

\begin{proof}
By Lemma~\ref{l.constraint},
\begin{equation}
  \int_\cD \hat \bw \cdot \sQ \hat \bu \, d \bx
  = \int_\cD \hat \bw \cdot \bOmega \,
    \hat g \, d \bx
  = s \int_{\bbT^2} \int_{-s^{-1}}^0 \bOmega \cdot \hat \bw \circ \Chi \,
    \hat g \circ \Chi \, d \bxi \,,
  \label{e.w-Qu-proof-1}
\end{equation}
where $\hat g \circ \Chi$ is given by \eqref{e.g}.  As it is
integrated against a mean-free vector field, the first term on the
right of \eqref{e.g} does not contribute to the integral
\eqref{e.w-Qu-proof-1}, so that
\begin{align}
  \int_\cD \hat \bw \cdot \sQ \hat \bu \, d \bx
  & = - s \int_{\bbT^2} \int_{-s^{-1}}^0 \bOmega \cdot \hat \bw \circ \Chi
      \int_{-s^{-1}}^\zeta \grad \cdot
        (\sS \sP \hat \bu \circ \Chi') \, d\zeta' \, d \bxi
      \notag \\
  & = s \int_{\bbT^2} \int_{-s^{-1}}^0
      \grad \cdot (\sS \sP \hat \bu \circ \Chi) \int_{-s^{-1}}^\zeta
      \bOmega \cdot \hat \bw \circ \Chi' \, d\zeta' \, d \bxi
    \notag \\
    & = - s \int_{\bbT^2} \int_{-s^{-1}}^0
      \hat \bu \circ \Chi \cdot \sP \sS \grad \int_{-s^{-1}}^\zeta
      \bOmega \cdot \hat \bw \circ \Chi' \, d\zeta' \, d \bxi
    \notag \\
    & = \int_\cD \hat \bu \cdot \sP \bC [\hat \bw] \, d \bx
   \label{e.wQuhat}
\end{align}
where $\bC[\hat \bw]$ is given by \eqref{e.Aw}.  We remark that the
second inequality is based on integration by parts in $\zeta$, the third
equality is due to the divergence theorem.  In both cases, the
boundary terms vanish due to the mean-free condition on $\hat \bw$.
We have further used the symmetry of the matrices $\sS$ and $\sP$.
\end{proof}

\begin{lemma}
\label{l.phi-bar-u3}
Let $\bu = \hat \bu + \bar \bu \in \Vdiv$ and let $\bar \phi$ be the
vertical mean of an arbitrary scalar field.  Then
\begin{equation}
  \int_\cD \bar \phi \, \bar u_3 \, d\bx
  = - \int_\cD \hat \bu \cdot \sP \grad \bar \phi \, z \, d \bx \,.
  \label{e.phi-bar-u3}
\end{equation}
\end{lemma}

\begin{proof}
By Lemma~\ref{l.constraint}, $\sQ \hat \bu = \bOmega \hat g$ with
\begin{equation}
  \hat g \circ \Chi (-s^{-1})
  = - s \int_{-s^{-1}}^0 \zeta \, \grad \cdot 
          (\sS \sP \hat \bu \circ \Chi) \, d \zeta \,.
\end{equation}
At the bottom boundary, $\bk \cdot \bar \bu = - \bk \cdot \hat \bu = -
\bk \cdot \sP \hat \bu - \bk \cdot \sQ \hat \bu$,
so that
\begin{align}
  \int_\cD \bar \phi \, \bar u_3 \, d\bx
  & = s \int_{\bbT^2}
    \int_{-s^{-1}}^0 \bar \phi \circ \Chi \,
      \bigl[
        - \bk \cdot \sP \hat \bu \circ \Chi (-s^{-1})
        - s \, \hat g \circ \Chi (-s^{-1})
      \bigr] \, d\bxi
    \notag \\
  & = \int_{\bbT^2} \bar \phi \circ \Chi \,
      \biggl[
        - \bk \cdot \sP \hat \bu \circ \Chi (-s^{-1}) 
        + s^2 \int_{-s^{-1}}^0 \zeta \,
          \grad \cdot (\sS \sP \hat \bu \circ \Chi) \, d \zeta
      \biggr] \, d x
    \notag \\
  & = - \int_{\bbT^2} \bar \phi \circ \Chi \, \bk \cdot \sP \hat \bu
    \circ \Chi (-s^{-1}) \, d x
    \notag \\
  & \quad + s^2 \int_{\bbT^2} \int_{-s^{-1}}^0 \bar \phi \circ \Chi \, 
       \bigl( \grad \cdot (\zeta \, \sS \sP \hat \bu \circ \Chi)
         - \grad \zeta \cdot \sS \sP \hat \bu \circ \Chi \bigr) \, d \bxi \,.
\end{align}
Since $\grad \zeta = \bk$, the second term in the last integral
vanishes as the vertical integration is over a mean free quantity.
For the first term in the last integral, we integrate by parts.  The
boundary term from the upper boundary is zero.  The boundary term from
the lower boundary exactly cancels the integral on the second last
line, so that
\begin{align}
  \int_\cD \bar \phi \, \bar u_3 \, d\bx
  & = - s^2 \int_{\bbT^2} \int_{-s^{-1}}^0 \hat \bu \circ \Chi \cdot \sP \sS
      \grad (\bar \phi \circ \Chi) \, \zeta \, d \bxi
    \notag \\
  & = - s \int_{\bbT^2} \int_{-s^{-1}}^0 \hat \bu \circ \Chi \cdot \sP \sS
      (\sA^\sT \grad \bar \phi \, z) \circ \Chi \, d \bxi \,,
\end{align}
where the last equality is due to
$\grad (f \circ \Chi) = (\sA^\sT \grad f) \circ \Chi$ and $z=s \zeta$.
Noting that $\sP \sS \sA^\sT = \sP$ and changing back to Cartesian
coordinates, we obtain \eqref{e.phi-bar-u3}.
\end{proof}

\begin{lemma}
\label{l.div-u-bar}
Under the conditions of Lemma~\ref{l.div-uhat}, there exists $\bar
\bu$, also divergence free, such that $\bu = \hat \bu + \bar \bu$
satisfies the zero-flux boundary condition $\bk \cdot \bu = 0$ at
$z=0,-1$. 
\end{lemma}


\begin{proof}
By Lemma~\ref{l.div-uhat}, $\hat \bu$ is divergence free. We now
choose $\bar u_3$ such that the vector field $\bu$ is tangent to the
top and bottom boundaries.  At $z=-1,0$, we require
\begin{align}
    \bk \cdot \bar \bu = - \bk \cdot \hat \bu \,.
    \label{e.u-boundary-condition}
\end{align}
Observe that
\begin{align}
\label{e.u3-bar-top}
  \hat u_3 \circ \Chi(0)
  & = \bk \cdot \sP \hat \bu \circ \Chi (0)
      + s \, \hat g \circ \Chi (0)
    \notag \\
    & = \bk \cdot \sP \hat \bu \circ \Chi (0)
      + s \, \hat g \circ \Chi (-s^{-1})
      - s \int_{-s^{-1}}^0 \grad \cdot
          (\sS \sP \hat \bu \circ \Chi) \, d\zeta
    \notag \\
    & = \bk \cdot \sP \hat \bu \circ \Chi (0)
      + s \, \hat g \circ \Chi (-s^{-1})
      - \bk \cdot \sP \hat \bu \circ \Chi(0)
      + \bk \cdot \sP \hat \bu \circ \Chi(-s^{-1})
    \notag \\
    & = \bk \cdot \sP \hat \bu \circ \Chi(-s^{-1})
      + s \, \hat g \circ \Chi (-s^{-1}) \,.
\end{align}
This implies $\hat u_3$ takes the same value at
the bottom and top boundaries along any line in the direction of the
axis of rotation.  Consequently, we can use
\eqref{e.u-boundary-condition}  to \emph{define} $\bk \cdot \bar \bu$
at the bottom, i.e., we set 
\begin{equation}
  \label{e.u3-bar-bottom}
  \bar u_3 \circ \Chi
  = - \bk \cdot \sP \hat \bu \circ \Chi (-s^{-1})
    - s \, \hat g \circ \Chi (-s^{-1}) \,.
\end{equation}
Then the boundary condition \eqref{e.u-boundary-condition} is
satisfied at $z=0$ as well.

We now choose $\bar u$ such that $\bu$ is divergence-free.  Indeed,
due to Lemma~\ref{l.mean-divergence2}, it suffices to ensure that
$\nabla \cdot \sS_h \sP \bar \bu = 0$, cf.\
\eqref{e.assumption-coords} for an explicit expression.  Setting
$\bar u = \nabla \phi$, we see that this implies
\begin{equation}
  \Delta \phi = \frac{c}s \, \partial_y \bar u_3 \,,
\end{equation}
which can be solved as a Poisson equation on $\bbT^2$.
\end{proof}

\begin{lemma}
\label{l.uJv}
Let $\bu, \bv \in \Vdiv$ be such that their domain-mean is zero.  Then
\begin{equation}
  \int_\cD \bu \cdot \sJ \bv \, d\bx
  = \int_\cD \hat \bu \cdot \sJ \hat \bv \, d\bx \,.
  \label{e.bar-v-vanishes}
\end{equation}
\end{lemma}

\begin{proof}
The mean of $\bu$ over the domain $\cD$ is zero if and only if the
horizontal mean of $\bar \bu$ is zero, and likewise for $\bv$.
Moreover, by Lemma~\ref{l.mean-divergence1},
$\nabla \cdot (\sS_h \sP \bar \bu)=0$ and
$\nabla \cdot (\sS_h \sP \bar \bv)=0$.  Thus, there exist scalar
fields $\psi$ and $\theta$ such that
$\sS_h \sP \bar \bu = \nabla^\bot \psi$ and
$\sS_h \sP \bar \bv = \nabla^\bot \theta$. We write
\begin{equation}
  \sS^{-1}_h
  = \begin{pmatrix}
      1 & 0 \\ 0 & s^2 \\ 0 & - cs
    \end{pmatrix}
  \label{e.pseudoinverse}
\end{equation}
to denote the pseudo-inverse of $\sS_h$ on $\Range \sP$. It satisfies
\begin{subequations}
\begin{gather}
  \sS_h \, \sS_h^{-1} = \sI_{2} \,,
\intertext{where $I_2$ denotes the $2 \times 2$ identity matrix, and}
  \sS_h^{-1} \, \sS_h \, \sP = \sP \,.
\end{gather}
\end{subequations}
Then, $\sP \bar \bu = \sS_h^{-1} \, \nabla^\bot \psi$ and
$\sP \bar \bv = \sS_h^{-1} \, \nabla^\bot \theta$. Further, observing that
\begin{gather}
  \sS_h^{-\sT} \, \sJ \, \sS_h^{-1} = s \, \sJ_{2} \,,
\end{gather}
where $\sJ_{2}$ denotes the canonical $2 \times 2$ symplectic matrix, and
recalling that $\sP \sJ = \sJ = \sJ \sP$, we compute
\begin{align}
    \int_\cD \bar \bu \cdot \sJ \bar \bv \, d\bx
    & = \int_\cD \sP \bar \bu
          \cdot \sJ \sP \bar \bv \, d\bx
    \notag \\
    & = \int_\cD \sS_h^{-1} \nabla^\bot \psi
          \cdot \sJ \sS_h^{-1} \nabla^\bot \theta \, d\bx
    \notag \\
    & = \int_\cD \nabla^\bot \psi
          \cdot \sS_h^{-\sT} \sJ \sS_h^{-1} \nabla^\bot \theta \, d\bx
    \notag \\
    & = - s \int_\cD \nabla^\bot \psi
          \cdot \nabla \theta \, d\bx \,.
    \label{e.skew-inner-product}
\end{align}
By orthogonality of gradients and curls, the last integral is zero,
which implies \eqref{e.bar-v-vanishes}.
\end{proof}

\section*{Appendix D. Derivation of potential energy contribution to $L_1$}
\label{a.pe-derivation}
\renewcommand{\theequation}{D.\arabic{equation}}
\setcounter{equation}{0}

In the following, we give a detailed derivation of the potential
energy contribution to the $L_1$-Lagrangian.  Inserting the boundary
condition for the transformation vector field \eqref{e.vn3-bar} and
the representation of $\sQ \bar \bv$ via \eqref{e.g2}, we compute:
\begin{align}
\label{e.l1-PE2}
  - \int_\cD \rho \, \bv \cdot \bk \, d\bx
  & = - s \int_{\bbT^2} \int_{-s^{-1}}^0 \rho \circ \Chi \,
      \bigl(
        \bk \cdot \bar \bv \circ \Chi
        + \bk \cdot \sP \hat \bv \circ \Chi
        + \bk \cdot \sQ \hat \bv \circ \Chi
      \bigr) \, d\bxi
    \notag \\
  & = s \int_{\bbT^2} \int_{-s^{-1}}^0 \rho \circ \Chi \,
        \biggl[
          \bk \cdot \sP \hat \bv \circ \Chi (-s^{-1})
          + s \, \hat g \circ \Chi (-s^{-1})
          - \bk \cdot \sP \hat \bv \circ \Chi
        \notag \\
        & \quad
          - s \, \hat g \circ \Chi (-s^{-1})
          + s \int_{-s^{-1}}^\zeta \grad \cdot
            (\sS \sP \hat \bv \circ \Chi')\, d \zeta'
        \biggr] \, d\bxi
    \notag \\
  & = s \int_{\bbT^2} \int_{-s^{-1}}^0 \rho \circ \Chi \,
        \biggl[
          \bk \cdot \sP \hat \bv \circ \Chi (-s^{-1})
          - \bk \cdot \sP \hat \bv \circ \Chi  
    \notag \\
    & \quad
          + \int_{-s^{-1}}^\zeta
            \partial_\zeta (\bk \cdot \sP \hat \bv \circ \Chi') \, d\zeta'
          + s \int_{-s^{-1}}^\zeta
            \nabla \cdot (\sS_h \sP \hat \bv \circ \Chi') \, d\zeta'
        \biggr] \, d\bxi
    \notag \\
    & = - s^2 \int_{\bbT^2} \int_{-s^{-1}}^0
          \nabla \rho \circ \Chi \cdot \int_{-s^{-1}}^\zeta
          \sS_h \sP \hat \bv \circ \Chi' \, d \zeta' \, d\bxi
\end{align}
Further, inserting \eqref{e.Pv-hat}, using the identity
\begin{gather}
\label{e.identity-SJ}
  \sS_h \, \sJ^\sT = - s^{-1} \, \sJ_2 \, \sP_h \,,
\end{gather}
and noting that horizontal gradients and composition with $\Chi$
commute, we obtain
\begin{align}
  - \int_\cD \rho \, \bv \cdot \bk \, d\bx 
  & = - s^2 \int_{\bbT^2} \int_{-s^{-1}}^0
        \nabla \rho \circ \Chi \cdot \int_{-s^{-1}}^\zeta
        \sS_h \sJ^\sT \, \bigl( - \tfrac{1}{2} \, \hat \bV +
        \lambda \, \hat \bV_g \bigr) \circ \Chi' \, d\zeta' \, d\bxi
    \notag \\
  & = s \int_{\bbT^2} \int_{-s^{-1}}^0
        \nabla \rho \circ \Chi \cdot \int_{-s^{-1}}^\zeta
        \sJ_2 \, \sP_h \, \bigl( - \tfrac{1}{2} \, \hat \bV +
        \lambda \, \hat \bV_g \bigr) \circ \Chi' \, d\zeta' \, d\bxi
    \notag \\
  & = s \int_{\bbT^2} \int_{-s^{-1}}^0
        \nabla^\bot \rho \circ \Chi \cdot \int_{-s^{-1}}^\zeta
        \sP_h \, \bigl( \tfrac{1}{2} \, \hat \bV -
        \lambda \, \hat \bV_g \bigr) \circ \Chi' \, d\zeta' \, d\bxi \,.
\end{align}
Inserting the thermal wind relation \eqref{e.twr-mean}, integrating by
parts with respect to $\zeta$, and changing variables, we continue the
computation:
\begin{align}
  - \int_\cD \rho \, \bv \cdot \bk \, d\bx
  & = - s \int_{\bbT^2} \int_{-s^{-1}}^0
        \partial_\zeta (\hat u_g \circ \Chi) \cdot
        \int_{-s^{-1}}^\zeta \sP_h \, \bigl( \tfrac{1}{2} \, \hat \bV -
        \lambda \, \hat \bV_g \bigr) \circ \Chi' \, d\zeta' \, d\bxi
    \notag \\
  & = \int_\cD
        \hat u_g \cdot \sP_h \bigl( \tfrac{1}{2} \, \hat \bV -
        \lambda \, \hat \bV_g \bigr) \, d\bx 
    \notag \\ 
  & = \int_\cD \hat u_g \cdot 
        \bigl( \tfrac{1}{2} \, \hat u - \lambda \, \hat u_g \bigr) \, d\bx
\label{e.PE-2}
\end{align}
where, in the last step, we have made use of Lemma~\ref{l.w-Qu},
Lemma~\ref{l.phi-bar-u3}, and the fact that $\bk \cdot \bu_g = 0$.

\bibliographystyle{apacite}
\renewenvironment{APACrefURL}[1][]{}{}
\def\url#1{}
\bibliography{trr}

\end{document}